\documentclass[12pt,reqno]{amsart}
\usepackage[utf8x]{inputenc}
\usepackage[T1]{fontenc}
\usepackage{mathrsfs,amsmath,amsxtra,amssymb,amsthm,amsfonts}
\renewcommand{\S}{\mathcal{S}}
\newcommand{\Z}{\mathbb{Z}}
\newcommand{\C}{\mathbb{C}}
\newcommand{\T}{\mathbb{T}}
\newcommand{\bs}{{\bf s}}
\newcommand{\bx}{{\bf x}}
\newcommand{\bt}{{\bf t}}

\renewcommand{\for}{\begin{eqnarray*}}
\newcommand{\mel}{\end{eqnarray*}}

\newcommand{\ten}{\otimes}

\newcommand{\pl}{\hspace{.1cm}}
\newcommand{\pll}{\hspace{.3cm}}

\newcommand{\ran}{\rangle}
\newcommand{\lan}{\langle}

\newcommand{\al}{\alpha}
\newcommand{\si}{\sigma}

\newcommand{\la}{\lambda}

\newcommand{\A}{{\mathcal A}}

\newcommand{\K}{{\mathcal K}}

\newcommand{\R}{\mathbb{R}}

\renewcommand{\i}{\subset}

\newcommand{\norm}[2]{\parallel \! #1 \! \parallel_{#2}}

\newcommand{\pd}[1]{\frac{\partial}{\partial {#1}}}

\newtheorem{lemma}{Lemma}[section]
\newtheorem{prop}[lemma]{Proposition}
\newtheorem{theorem}[lemma]{Theorem}
\newtheorem*{theorem*}{Theorem}
\newtheorem{cor}[lemma]{Corollary}

\newtheorem{rem}[lemma]{Remark}

\newcommand{\re}{\begin{rem}\rm}
\newcommand{\mar}{\end{rem}}

\newcommand{\bra}[1]{\langle{#1}|}
\newcommand{\ket}[1]{|{#1}\rangle}

\newcommand{\qd}{\end{proof}\vspace{0.5ex}}
\newcommand{\prf}{\begin{proof}[\bf Proof:]}

\newcommand{\xspace}{\hbox{\kern-2.5pt}}

\allowdisplaybreaks

\begin{document}
\title{Continuous perturbations of Noncommutative Euclidean spaces and tori}
\author{Li Gao}
\address{Department of Mathematics\\
University of Illinois, Urbana, IL 61801, USA} \email[Li Gao]{ligao3@illinois.edu}
\begin{abstract}
We prove a noncompact version of Haagerup and R{\o}rdam's result about continuous paths of the rotation $C^*$-algebras. It gives a continuous Moyal deformation of Euclidean plane. %This result is also an example of the continuous embedding of a non-exact continuous field of $C^*$-algebra over a noncompact space.
Moveover, the construction is generalized to noncommutative Euclidean spaces of dimension $d\ge 2$. As a corollary, we obtain Lip$^{\frac12}$ continuous maps for the generators of noncommutative $d$-tori.
\end{abstract}
%\author[M. Junge]{M. Junge$^*$}\thanks{$^*$ Partially supported by NSF-DMS 1201886 and 1501103}
%\address{Department of Mathematics\\
%University of Illinois, Urbana, IL 61801, USA} \email[Marius
%Junge]{mjunge@illinois.edu}
%\renewcommand{\abstractname}{\bf{Abstract}}
\maketitle
\section{Introduction} The celebrated Heisenberg commutation relation, \[PQ-QP=-iI\pl,\]
where $I$ is the identity operator, plays an important role in quantum mechanics and the related mathematics. This commutation relation affiliates to the Moyal deformation of Euclidean plane. Let $d\ge 2$ and $\theta=(\theta_{jk})_{j,k=1}^{d}$ be a real skew-symmetric $d\times d$-matrix. The associated noncommutative Euclidean space (for a nonsingular $\theta$) is given by $d$ one-parameter unitary groups $u_1(t), u_2(t),\cdots, u_d(t)$ satisfying the following commutation relations
\begin{align*}
u_j(s)u_k(t)=e^{i st\theta_{jk} }u_k(t)u_j(s)\pl, \pl \forall \pl\pl s,t \in \R   \pl,
\end{align*}
for $j,k=1,2,\cdots, d$. This noncommutative space, also called Moyal plane, is a prototype of noncompact noncommutative manifolds (see e.g. \cite{moyalplane}). Moreover, interesting objects and structures from quantum physics have been studied on the noncommutative plane and noncommutative $\R^4$ (see e.g. \cite{instantons, nair01, smailagic03}).

Another class of fundamental examples in noncommutative geometry are the noncommutative tori. They have been extensively studied over decades (we refer to the survey paper by Rieffel \cite{Rieffel90} for the study before 90s, and \cite{xu12, Connes14, Elliott07} for more recent {development).} Recall that the noncommutative $d$-torus $A^d_\theta$ associated to a skew-symmetric matrix $\theta$ is the universal $C^*$-algebra generated by $d$ unitaries $u_1,u_2,\cdots,u_d$ subject to the commutation relations
\begin{align*}u_ju_k=e^{2\pi i \theta_{jk} }u_ku_j\pl, \pl j ,k=1, 2,\cdots ,d  \pl.
\end{align*}
It is clear from the definition that $A^d_\theta$ is a noncommutative deformation of $C(\T^d)$, the $\text{$C^*$-algebra}$ of continuous functions on a usual $d$-torus ($\theta=0$). When $d=2$, the commutation relations reduce to two unitaries $u,v$ satisfying
 \begin{align*}uv=e^{2\pi i \theta }vu \end{align*}
 for a real number $\theta$. The noncommutative $2$-tori are also called rotation $C^*$-algebras (cf. \cite{Davidson}).

In this paper, we will consider generalizations of the following result by Haagerup and R{\o}rdam in \cite{HR}.
\begin{theorem*}[Haargerup-R{\o}dam, $'95$]\label{HR}Let $H$ be an infinite dimensional Hilbert space and $U(H)$ be its unitary group. There exist two continuous paths ${u,v: [0,1] \to U(H)}$ and a universal constant $C>0$ such that $u(0)=u(1), v(0)=v(1)$, and
\begin{enumerate}
\item $u(\theta)v(\theta)=e^{2\pi i\theta}v(\theta) u(\theta)$;
\item $\max \{\norm{u(\theta)-u(\theta ')}{},\pl \norm{v(\theta)-v(\theta ')}{}\}\le  C|\theta-\theta '|^{\frac{1}{2}}$;
\end{enumerate}
 for all $\theta, \theta'\in [0,1]$.
\end{theorem*}
It is shown by Elliott in \cite{Elliott83} that the family of rotation algebras forms a continuous field of $C^*$-algebra (see e.g. \cite{Fillmore96} for the definition). The above theorem gives this family a continuous embedding in $B(H)$. Kirchberg and Phillips \cite{Kirchberg} also obtain a Lip$^{\frac12}$ continuous embedding of rotation algebras into the Cuntz algebra $\mathcal{O}_2$. The existence of Lip$^{\frac12}$ continuous paths has applications in estimating the spectrum of almost Mathieu operators (see \cite{Boca}). For higher dimension, both Theorem 5.7 of \cite{Kirchberg} and Theorem 3.2 of \cite{bl} prove the existence of the continuous embedding in norm for noncommutative $d$-tori, but with little information about the concrete continuity.

We show that the Lip$^{\frac12}$ continuous maps also exist for noncommutative $d$-tori of dimension $d> 2$. Let us denote $\A[d]\equiv [0,1]^{\frac{(d-1)d}{2}}$ as the space of all skew-symmetric $d\times d$ matrices with entries in the unit interval.
\begin{theorem}\label{p}
There exist $d$ continuous maps $u_1,u_2, \cdots, u_d: \A[d] \to U(H)$ and a universal constant $C>0$ such that \begin{enumerate}
 \item[i)] $u_ju_k=e^{2\pi i \theta_{jk}}u_ku_j\pl, \pl j,k=1, 2,\cdots,d\pl ;$
  \item[ii)]$\displaystyle \norm{u_j(\theta)-u_j(\theta')}{}\le C(\sum_{1\le k\le d}|\theta_{jk}-\theta'_{jk}|^{\frac{1}{2}})\pl,\pl j=1, 2,\cdots,d\pl ;$
  \end{enumerate}
for all $\theta, \theta'\in  \A[d]$.
\end{theorem}
The proof is based on an explicit construction, which we illustrate here for the case $d=3$. Given the two continuous paths $u,v$ from Theorem \ref{HR}, we define the following maps $u_1, u_2, u_3: \A[3]\to U(H^{\ten 3})$,
\begin{align}&u_1(\theta)=u(\theta_{12})\ten u(\theta_{13}) \ten I\pl ,\pl  u_2(\theta)=v(\theta_{12})\ten I\ten u(\theta_{23})\pl,  \nonumber\\ &u_3(\theta)=I\ten v(\theta_{13})\ten v(\theta_{23})\pl . \label{tensor}
\end{align}
Because each pair of operators only shares one nontrivial tensor component (other than the identity), $u_1, u_2$ and $u_3$ satisfy the commutation relations
\begin{align*}
& u_1(\theta)u_2(\theta)= e^{2\pi i \theta_{12}}u_2(\theta)u_1(\theta) \pl, \pl u_1(\theta)u_3(\theta)= e^{2\pi i \theta_{13}}u_3(\theta)u_1(\theta) \pl,  \\ & u_2(\theta)u_3(\theta)= e^{2\pi i \theta_{23}}u_3(\theta)u_2(\theta) \pl, \end{align*}
By induction, this construction can be generalized to higher dimension and the Lip$^{\frac12}$ continuity follows from the triangle inequality. It also works for the paths in the Cuntz algebra $\mathcal{O}_2$ and implies a continuous embedding into $\mathcal{O}_2$ because $\mathcal{O}_2\ten \mathcal{O}_2=\mathcal{O}_2$ (see \cite{Kirchberg}).

We also prove the ``noncompact'' analog of the above results. It is proved in \cite{HR} that for an infinite multiplicity representation $(P, Q)$ of the Heisenberg relation, there exists a commuting pair of self-adjoint operators $(P_0, Q_0)$ on $H$ such that $P-P_0$ and $Q-Q_0$ are bounded. This bounded perturbation of unbounded operators is used in the construction of continuous path of rotation algebras. We find that the methods of Haagerup and R{\o}rdam, with careful modifications, also applies to the Heisenberg relation, to construct a continuous Moyal deformation of $\R^2$. Moreover, using the same idea of \eqref{tensor}, this can be generalized to noncommutative Euclidean space of dimension $d > 2$.

Denote $\A(d)\equiv \mathbb{R}^{\frac{d(d-1)}{2}}$ as the space of all real skew-symmetric $d\times d$-matrices. Our main result can be stated as follows.
\begin{theorem}
\label{ultimate}
There exist continuous maps $u_1,u_2,\cdots u_d: \A(d) \times \mathbb{R} \to U(H)$ and a universal constant $C>0$ such that
\begin{enumerate}
\item[i)] for each $\theta$, $u_1(\theta,\cdot),u_2(\theta,\cdot),\cdots, u_d(\theta,\cdot)$ are
 strongly continuous one-parameter unitary groups satisfying
    \[u_j(\theta,s)u_k(\theta,t)=e^{ist\theta_{jk}}u_k(\theta,t)u_j(\theta,s)\pl, \pl \pl \forall \pl s,t \in \R\pl , \pl j,k=1,\cdots, d \pl;\]
\item[ii)]for any $t\in \R$ and $\theta, \theta'\in \A(d)$,
\[ \norm{u_j(\theta,t)-u_j(\theta',t)}{}\le C|t|(\sum_k |\theta_{kj}-\theta_{kj}'|^{\frac12}) \pl, \pl  j=1,\cdots, d \]
\end{enumerate}
\end{theorem}

The present work is organized as follows. In Section 2, we apply the method of Haagerup and R{\o}rdam to construct a continuous deformation of the Heisenberg relations. Section 3 extends both the bounded perturbation from \cite{HR} and the continuous deformation to noncommutative Euclidean space of higher dimension.  Section 4 is devoted to corresponding results for noncommutative $d$-tori.

\section{Continuous Perturbation of Heisenberg relations} We first discuss Theorem \ref{ultimate} for $d=2$ by the method of Haagerup of R{\o}rdam in  \cite{HR}. In this section,  $\theta$ always denotes a real number. Let $P$ and $Q$ be two (unbounded) operators on a Hilbert space $H$ and $u(s)=e^{iPs}, v(t)=e^{iQt}$ be their associated one-parameter groups.
For a nonzero $\theta$, we say $P$ and $Q$ satisfy \emph{the Heisenberg relation with parameter $\theta$} \begin{align}[P,Q]=PQ-QP=-i\theta I \pl, \label{Heisenberg}\end{align}
if $u(s), v(t)$ satisfy the Weyl relation%(with paramter $\theta$)
\begin{align}u(s)v(t)=e^{ist\theta }v(t)u(s)\pl. \label{weyl}\end{align}
%The phase change in \eqref{weyl} is $\theta$-amplified than the usual Weyl relations.
When $\theta=1$, we call \eqref{Heisenberg}  \emph{the standard Heisenberg relation}.
Thanks to the well-known von Neumann-Stone theorem (c.f. \cite{Hall} pp. 285-287), any representation of the standard Heisenberg relation is unitarily equivalent to a (finite or infinite) multiple of the Schr{\"o}dinger picture. More precisely, the only irreducible representation, up to a unitary equivalence, is given by the momentum operator and position operator from quantum mechanics
\[Pf=-i\frac{df}{dx}\pl, \pl (Qf)(x)=xf(x)\pl ,\pl f\in C_c^1(\R)\pl.\]
Both $P, Q$ are unbounded self-adjoint operators on the Hilbert space $L_2(\mathbb{R})$ and they have a common core of $C^1_c(\R)$ (continuously differentiable compactly supported functions).
%For example, the $\theta$ in the quantum mechanic model is the ( reduced) Planck constant $\hbar$.
The associated one-parameter unitary groups are given by
\begin{align}(u(s)f)(x)=f(x+s)\pl,  \pl (v(t)f)(x)=e^{ixt}f(x)\pl, \label{action}\end{align}
which satisfy \eqref{weyl}. Note that \eqref{Heisenberg} implies that
$(\frac{1}{\theta}P, Q)$ satisfies the standard Heisenberg relation, then the Stone-von Neumann Theorem is easily generalized for any nonzero $\theta$.
When $\theta=0$, the one-parameter groups commute
\[e^{iPs}e^{iQt}=e^{iQt}e^{iPs} \pl,\]
and we say $P$ and $Q$ commute \emph{strongly}. Strongly commuting pairs $(P,Q)$ are one-to-one corresponding to unitary representations of $\R^2$. In particular, the left regular group representation of $\R^2$ is given by
\begin{align}\label{R2}Pf=-i\frac{df}{dx}\pl, \pl Q=-i\frac{df}{dy}\pl
\end{align}
as unbounded self-adjoint operators on $L_2(\R^2)$ and the unitaries are translations,
\begin{align*}
(u(t)f)(x,y)= f(x+t,y)\pl ,\pl (v(t)f)(x,y)= f(x,y+t) \pl, \pl f\in L_2(\R^2)\pl.
\end{align*}
%We will frequently omit the explicit action of above operator  but assume the action as above.
%The von Neumann algebra generated by these two commuting unitary groups is $L(\R^2)\cong L_\infty(\R^2)$, whereas for a nonzero $\theta$, $u(t),v(t)$ generate $B(L_2(\R))$.
We will combine the discussions of Section 3 and 5 from \cite{HR}, working on the unbounded operators $(P,Q)$ instead of unitaries. Let us begin with a modification of Theorem 3.1 of \cite{HR}. %The main idea of proof remains the same. We includes a relative complete proof for the convenience of reader.
\begin{theorem}\label{bounded}Let $\theta\neq 0$. Let $(P,Q)$ be an representation of $[P,Q]=-i\theta I$ with infinite multiplicity on a separable Hilbert space $H$. Then for any $\theta' \in \R$, there exist self-adjoint operators $P'$ and $Q'$ on $H$ satisfying $[P', Q']=-i\theta'I$ such that $P-P'$ and $Q-Q'$ are bounded and moreover,
\[\max \{\norm{P-P'}{}\pl ,\norm{Q-Q'}{}\}\le 9|\theta -\theta'|^{\frac{1}{2}}\pl .\]
\end{theorem}
\begin{proof} We may first assume $\theta'>\theta$ and denote $\delta= |\theta'-\theta|^{\frac12}$. Let $K$ be an infinite dimensional separable Hilbert space. Because all infinite multiplicity representations of the Heisenberg relation on a separable Hilbert space are unitarily equivalent,
we may assume that $(P,Q)$ is given by
\[P=-i\delta\frac{\partial }{\partial x}\pl, \pl Q = -i\frac{\partial}{\partial y}+\frac{\theta}{\delta}  x \pl,\]
on $L_2(\R^2,K)$. The associated one-parameter groups $u(t)=e^{iPt}, v(t)=e^{iQt}$ are
\[(u(t)f)(x,y)=f(x+\delta t, y)\pl, (v(t)f)(x,y)=e^{i\frac{\theta}{\delta} x t}f(x, y+ t)\pl.\]
Let $w: \mathbb{R}^2 \to U(K)$ be a $C^1$-function with values in the unitary group $U(K)$ of $K$. It can be regarded as a unitary on $L_2(\mathbb{R}^2, K)$ via pointwise action
\[(wf)(x,y)=w(x,y)f(x,y) \pl.\]
The subspace $C^1_c(\R^2, K)$ is a common core of $P, Q$ and also invariant under $w$. Then for $f \in C^1_c(\R^2, K)$,
\begin{align}\label{cal}(w^*Pw f)(x,y)&= -i\delta(\frac{\partial f}{\partial x}(x,y)+iw(x,y)^*\frac{\partial w}{\partial x}(x,y) f(x,y))\pl ,  \nonumber \\(w^*Qw f)(x,y)&= -i(\frac{\partial f}{\partial y}(x,y)+iw(x,y)^*\frac{\partial w}{\partial y}(x,y)f(x,y))+\frac{\theta}{\delta} xf(x,y) \pl .
\end{align}
It is proved in Theorem 3.1 of \cite{HR} that there exists a $C^1$-function $w: \mathbb{R}^2 \to U(K)$ such that
\begin{align}\sup_{(x,y)\in \R^2}\norm{\frac{\partial w}{\partial x}(x,y)}{}\le 9\pl, \pl \sup_{(x,y)\in \R^2}\norm{\frac{\partial w}{\partial y}(x,y)-ixw(x,y)}{}\le 9\pl. \label{w}
\end{align}
Set $\bar{w}(x,y)=w(x,\delta y)$, and choose the self-adjoint operators $P'= \bar{w}P\bar{w}^*\pl , Q'=\bar{w}(Q+\frac{\theta'-\theta}{\delta} x)\bar{w}^*$. The pair $(P', Q')$ satisfies $[P', Q']=-i\theta' I$ and also shares the common core $C^1_c(\R^2, K)$. On this dense domain $C^1_c(\R^2, K)$
\begin{align*}
P-P'&=\bar{w}(\bar{w}^*P\bar{w}- P)\bar{w}^*=i\bar{w}(\delta \bar{w}^*\frac{\partial \bar{w}}{\partial x})\bar{w}^*=i\delta \frac{\partial \bar{w}}{\partial x}\bar{w}^*\pl ,\\
Q-Q'&= \bar{w}(\bar{w}^*Q\bar{w}- Q-\delta x)\bar{w}^*= i\bar{w}(\bar{w}^*\frac{\partial \bar{w}}{\partial y}-\delta x)\bar{w}^*=i(\frac{\partial \bar{w}}{\partial y}\bar{w}^*-\delta x) \pl.
\end{align*}
Both are bounded because for any $(x,y)\in \R^2$,
\begin{align*}
\norm{-i \frac{\delta\partial \bar{w}}{\partial x}(x,y)\bar{w}(x,y)^*}{}&= \norm{\delta \frac{\partial w}{\partial x}(x, \delta y)}{} \le 9\delta \pl , \\
\norm{-i\frac{\partial \bar{w}}{\partial y}(x,y)\bar{w}^*(x,y)-\delta x}{}&= \norm{\delta(\frac{\partial w}{\partial y}(x, \delta y)-ixw(x, \delta y))}{}\le 9\delta \pl.
\end{align*}
For $\theta>\theta'$, the estimates follow similarly by taking $\bar{w}(x,y)=\bar{w}(x, -\delta y)$.
\end{proof}
\begin{rem}{\rm The pair $(P',Q')$ gives a representation of infinite multiplicity. In particular when $\theta'= 0$, $P'$ and $Q'$ strongly commute and are unitarily equivalent to an infinite multiple of regular representations \eqref{R2}. Conversely, the above theorem remains valid for $\theta=0$ if in addition $(P,Q)$ is unitarily equivalent to the regular representation. \label{sc}}
\end{rem}
Stone's theorem states that self-adjoint operators on a Hilbert space $H$ are  one-to-one correspondent to one-parameter unitary groups in $B(H)$. The next proposition shows that this correspondence is of certain bi-continuity.
\begin{prop} \label{eq}
Let $P$ and $P'$ be (possibly unbounded) self-adjoint operators on a
Hilbert space $H$. Then their domains coincide $D(P)=D(P')$ and $P-P'$ is bounded with its norm less than a constant $C> 0$ if and only if $\norm{e^{iPt}-e^{iP't}}{}\le C|t|$ for any $t\in \R$.
\end{prop}
\begin{proof}The necessity is Lemma 4.3 of \cite{HR}. Here we prove the sufficiency. For $\xi\in D(P)$ and $\eta\in D(P')$, it follows by Stone's theorem that
\[\lim_{t\to 0}\frac{e^{iPt}\xi-\xi}{t}=iP\xi \pl, \pl \lim_{t\to 0}\frac{e^{iP't}\eta-\eta}{t}=iP'\eta\]
converge strongly. Then the derivative of the inner product $\lan e^{iPt}\xi,  e^{iP't}\eta\ran$ at $t=0$ is given by
\begin{align*}\lim_{t\to 0}\frac{1}{t}(\lan e^{iPt}\xi,  e^{iP't}\eta\ran-\lan\xi, \eta\ran)&=\lim_{t\to 0}\frac{1}{t}(\lan e^{iPt}\xi,  e^{iP't}\eta\ran-\lan e^{iPt}\xi, \eta\ran)+\frac{1}{t}(\lan e^{iPt}\xi, \eta\ran -\lan\xi, \eta\ran)\\
=&\lan\xi,iP'\eta\ran+\lan iP\xi,\eta\ran \pl.\end{align*}
On the other hand,
\[ \norm{e^{-iP't}e^{iPt}-1}{}=\norm{e^{iPt}-e^{iP't}}{} \le C|t| \pl, \pl t\in\R\] by assumptions. This implies
\begin{align*}
\lan e^{iPt}\xi,  e^{iP't}\eta\ran-\lan\xi, \eta\ran=\lan (e^{-iP't}e^{iPt}-1)\xi, \eta\ran\le Ct\norm{\xi}{}\norm{\eta}{}\pl.
\end{align*}
Therefore
\begin{align}
\label{diff}
|\lan\xi,iP'\eta\ran-\lan iP\xi,\eta\ran |\le C\norm{\xi}{}\norm{\eta}{} \pl,\pl |\lan\xi , P'\eta\rangle|\le  (C\norm{\xi}{}+\norm{P\xi}{})\norm{\eta}{} \pl.
\end{align}
Since $P'$ is self-adjoint, we have $\xi\in D(P'^*)=D(P')$. Now we are able to rewrite \eqref{diff} to obtain
\[|\lan (P-P')\xi,\eta\ran|\le C\norm{\xi}{}\norm{\eta}{}\]
 for all $\xi\in D(P),\eta\in D(P')$. Since $D(P)$ and $D(P')$ are dense in $H$, $\norm{P-P'}{}\le C$ and $P$ and $P'$ have the same domain. Note that for sufficiency we only use \[\norm{e^{iPt}-e^{iP't}}{}\le C|t|\pl , \pl t\in [0, \epsilon]\] for some $\epsilon>0$.
\end{proof}
The following is Lemma 5.1 of \cite{HR}, which is used as a key tool in the construction of continuous paths. We omit its proof here.
\begin{lemma}\label{up}
Let $M\subset B(H)$ be a von-Neumann algebra with properly infinite commutant $M'$. For a unitary $u\in M$, there exists a smooth path $u(t), t\in [0,1]$ of unitary, such that \begin{enumerate}
\item[i)]$u(0)=1$ and $u(1)=u$;
\item[ii)]$\norm{u'(t)}{}\le 9$;
\item[iii)]$\norm{[u(t),a]}{}\le 4\norm{[u,a]}{}$;
\item[iv)]$\norm{[u'(t),a]}{}\le 9\norm{[u,a]}{}$;
\item[v)]$\norm{\frac{d}{dt} u(t)au(t)^*}{}\le 45\norm{[u,a]}{}$;
\end{enumerate}
for all $t\in [0,1]$ and $a\in M$.
\end{lemma}
The next lemma is an analog of Lemma 5.2 of \cite{HR}. %With the help of Propsition \ref{eq}, we interpolates the unbounded infinitesimal generators instead of unitaries.
\begin{lemma}
Let $\theta \neq \theta'$ both be nonzero and $k\in \mathbb{N}$ be given. Let $(P(\theta),Q(\theta))$ (resp. $(P(\theta'),Q(\theta'))$) be a representation of ${[P, Q]=-i \theta I}$ (resp. ${[P,Q]=-i \theta'I}$) of infinite multiplicity on a separable Hilbert space $H$. Denote the associated one-parameter unitary groups as \[u_0(t)=e^{iP(\theta)t}\pl,  v_0(t)=e^{iQ(\theta)t}\pl, u_1(t)=e^{iP(\theta')t}\pl , v_1(t)=e^{iQ(\theta')t}.\]
Assume that the commutant of $\{u_0(t),u_1(t), v_0(t), v_1(t)\pl |\pl t\in \R\}$ is properly infinite and $P(\theta)-P(\theta')$ and $Q(\theta)-Q(\theta')$ are bounded. Denote
\[d=\max\{\norm{P(\theta)-P(\theta')}{},\norm{Q(\theta)-Q(\theta')}{}\}\pl,\]
and set
\[s_j=\theta+\frac{j}{k}(\theta'-\theta)\pl, \pl j=0,1,\cdots, k\pl\]
so that $s_0=\theta, s_k=\theta'$.
Then there exist pairs $(P(s_j),Q(s_j)), j=1,2,\cdots k-1$, of self-adjoint operators on $H$ such that
\begin{enumerate}
\item[i)] $(P(s_j),Q(s_j))$ satisfies the Heisenberg relation $[P(s_j), Q(s_j)]=-i s_jI$;
\item[ii)] $P(s_j)-P(s_{j+1})$ and $Q(s_j)-Q(s_{j+1})$ are bounded and
\[\max \{\norm{P(s_j)-P(s_{j+1})}{},\norm{Q(s_j)-Q(s_{j+1})}{}\}\le 1224({|\theta-\theta'|}/{k})^{\frac12}+45{d}/{k} \pl;\]
\item[iii)] the commutant of the one-paramter groups $\{e^{iP(s_0)t},   \cdots, e^{iP(s_k)t}, e^{iQ(s_0)t},\cdots,e^{iQ(s_k)t}\}$ is properly infinite.
\end{enumerate}\label{interpolates}
\end{lemma}
\begin{proof}
%Assume that $\theta'$ is not zero, otherwise switch $\theta$ and $\theta'$.%
We can decompose $H=H_1\ten H_2\ten H_3$ as a tensor product of three infinite dimensional Hilbert spaces $H_1,H_2$ and $H_3$. Moreover we may assume the four one-parameter groups $\{u_0(t),u_1(t), v_0(t), v_1(t)\}$ are in the subalgebra $B(H_1)\ten \mathbb{C}1_{B(H_2)}\ten \mathbb{C}1_{B(H_3)}$, since commutant of $\{u_0(t),u_1(t), v_0(t), v_1(t)\}$ is properly infinite. Also, $P(\theta), Q(\theta) , P(\theta')$ and $Q(\theta') $ can be regarded as operators on $H_1$ by identifying $P$ with $P\ten 1_{B(H_2)}\ten 1_{B(H_3)}$.

Denote $\delta=({|\theta'-\theta|}/{k})^{\frac{1}{2}}$ and set $\bar{P}(s_0)= P(\theta) , \bar{Q}(s_0)= Q(\theta)$. We can
apply $\text{Theorem \ref{bounded}}$ inductively to obtain $k$ pairs of self-adjoint operators $(\bar{P}(s_j), \bar{Q}(s_{j}))$ on $H_1$ satisfying i)  and
\begin{align}\label{inductive}\max\{ \norm{\bar{P}(s_{j})-\bar{P}(s_j+1)}{},\norm{\bar{Q}(s_{j})-\bar{Q}(s_j+1)}{}\}\le 9 \delta\pl.\end{align}
By the assumption of $d$ and the triangle inequality, we have
\begin{align}\label{tri1}
\max\{\norm{\bar{P}(s_k)-P(\theta')}{}, \norm{\bar{Q}(s_k)-Q(\theta')}{}\}\le 9k\delta+d\pl .
\end{align}
Note that $s_k=\theta'$, both pairs $(\bar{P}(s_k), \bar{Q}(s_k))$ and $(P(\theta'),Q(\theta'))$ are infinite multiplicity representations of the Heisenberg relation with the nonzero parameter $\theta'$. Then $P(\theta')=WP(s_k)W^*, Q(\theta')=WQ(s_k)W^*$ for some unitary $W\in B(H_1)\ten \mathbb{C}1 \ten \mathbb{C}1$.
Thus \eqref{tri1} implies
\[\max\{\norm{\bar{P}(s_k)-W\bar{P}(s_k)W^*}{},\norm{\bar{Q}(s_k)-W\bar{Q}(s_k)W^*}{}\}\le 9k\delta+d\pl .\]
 Denote $\bar{u}_{s_j}(t)=e^{t\bar{P}(s_k)i}, \bar{v}_{s_j}(t)=e^{t\bar{Q}(s_k)i}$. By Proposition \ref{eq},
\[ \max\{\norm{ [w,\bar{u}_{s_k}(t)] }{}, \norm{ [w,\bar{v}_{s_k}(t)] }{}\}\le   ( 9k\delta+d)|t| \pl, \pl  \pl t\in \R\pl. \]
For any $j=0,1,\cdots,k$, by \eqref{inductive} and the triangle inequality we have
\[ \max\{\norm{ [w,\bar{u}_{s_j}(t)] }{}, \norm{ [w,\bar{v}_{s_j}(t)] }{}\}\le   ( 27k\delta+d)|t| \pl,   \pl \pl t\in \R \pl.\]
All of the operators above are in the subalgebra $B(H_1)\ten \mathbb{C}1 \ten \mathbb{C}1$, which is of properly infinite commutant inside $B(H_1)\ten B(H_2)\ten \mathbb{C}1$. Hence we can apply Lemma \ref{up} for $W$, to obtain a path of unitary $W: [0,1] \to B(H_1)\ten B(H_2)\ten \mathbb{C}1$ such that $W(0)=1, W(1)=I$, and
\[\max \{\norm{ \frac{d}{ds}W(s)u_{s_j}(t)W^*(s)}{}\pl , \pl \norm{ \frac{d}{ds}W(s)v_{s_j}(s)W^*(s)}{}\}\le 45 (27k\delta+d)|t| \pl , \]
for all $j=0,1,\cdots k$ and $t\in \R$. Now for each $j$, set
\[u_{s_j}(t)=W(\frac{j}{k}) \bar{u}_{s_j}(t)W^*(\frac{j}{k})\pl,  \pl v_{s_j}(t)=W(\frac{j}{k}) \bar{v}_{s_j}(t)W^*(\frac{j}{k}) \pl, \]
 and self-adjoint operators $P({s_j}), Q({s_j})$ as the associated infinitesimal generators. We claim that this gives the desired construction.

First, $P({s_0})=P(\theta), Q({s_0})=Q(\theta)$ and
 $P({s_k})=P(\theta'), Q({s_k})=Q(\theta')$. Each pair $(P(s_j), Q(s_j))$ satisfies the commutation relations i). Moreover,
 \begin{align*}
\norm{u_{s_{j}}(t)-u_{s_{j+1}}(t)}{}\le &\norm{W(\frac{j+1}{k})( \bar{u}_{s_{j+1}}(t)-\bar{u}_{s_j}(t))W^*(\frac{j+1}{k})}{} \\&+ \norm{\int_{\frac{j}{k}}^{\frac{j+1}{k}}\frac{d}{ds}(W(s) \bar{u}_{s_j}(t)W^*(s))ds}{}\\ \le &9\delta |t|+45 (27\delta+\frac{d}{k})|t|=(1224\delta+45\frac{d}{k})|t| \pl,
\end{align*}
 and the same bound holds for $\norm{v_{s_{j+1}}(t)-v_{s_{j}}(t)}{}$. By Lemma \ref{eq}, we obtain that
\[\max \{\norm{P(s_{j})-P(s_{j+1})}{}, \norm{Q(s_{j})-Q(s_{j+1})}{}\}\le 1224\delta+45\frac{d}{k}\pl .\]
Finally, all unitary groups $u_{s_j}(t), v_{s_j}(t)$ belong to $B(H_1)\ten B(H_2)\ten \mathbb{C}1$ and hence the commutant is properly infinite.
\end{proof}
\begin{rem}{\rm Note that for $(P(\theta),Q(\theta))$, we only use the fact that Theorem \ref{bounded} applies. Hence the above theorem remains valid if $\theta$ is $0$ and $(P(\theta),Q(\theta))$ is unitarily equivalent to an infinite multiple of left regular representation. This point is used in the next theorem.}
\end{rem}
Let us denote $S(H)$ as all self-adjoint operators on the Hilbert space $H$. Based on the above lemma, we construct maps $P,Q: \R \to S(H)$ with continuously bounded perturbation. The next theorem is an analog of Lemma 5.3 and Theorem 5.4 of \cite{HR} for Heisenberg relations.
\begin{theorem}\label{d2}
Let $H$ be an infinite dimensional Hilbert space. Then there exist maps $P, Q: \mathbb{R} \to S(H)$ and a universal constant $C>0$ such that
\begin{enumerate}
\item[i)]$[P(\theta),Q(\theta)]=-i\theta I$;
\item[ii)] $P(\theta)-P(\theta')$ and $Q(\theta)-Q(\theta')$ are bounded on $H$ and moreover,
\begin{align}\max\{\norm{P(\theta)-P(\theta')}{},\norm{Q(\theta)-Q(\theta')} {}\}\le C|\theta-\theta'|^{\frac12}\pl;\label{liphalf}
\end{align}
\end{enumerate}
for all $\theta, \theta'\in \R$.
\end{theorem}
\begin{proof}
Set $k=8100$ and $\Gamma=\bigcup_{n=1}^{\infty} \Gamma_n $, where
\[\Gamma_n=\{\frac{j}{k^n}|\pl j\in \mathbb{Z},\pl |j|\le (n+1)k^n\}\pl.\]
%We first contruct the map $P,Q$ on the dense subset $\Gamma$.  The idea is to use the Lemma \ref{interpolates} inductively. We
Write $H=H_1\ten H_2 $ where both $H_1$ and $H_2$ are infinite dimensional. %We can choose a pair of strongly commuting operators $(P(0),Q(0))$ on $H_1$ such that $(P(0),Q(0))$ is unitarily equivalent to an infinite mutiple of the regular representation of $\R^2$ as in Remark \ref{sc}.
 Let $K$ be a separable infinite dimensional Hilbert space. We may assume $H_1=L_2(\R^2, K)$ and define the map $P,Q$ in $S(L_2(\R^2, K))$ for all integers $j\in \mathbb{Z}$ as follows
\begin{align*} &(P(0)f) (x,y)=-i\frac{\partial f}{\partial x}(x,y)\pl , \pl (Q(0)f)(x,y)=-i\frac{\partial f}{\partial y}(x,y)\pl ,\pl f \in C^1_c(\R^2, K)\pl, \\
&P(j+1)= w^* P(j)w\pl , \pl Q(j+1)=w^*(Q(j)+x)w\pl,
\end{align*}
where $w\in U(L_2(\R^2, K))$ is the unitary operator described in \eqref{w}. $\text{Theorem \ref{bounded}}$ implies
\[\max \{\norm{P(j) -P(j+1)}{}, \norm{Q(j) -Q(j+1)}{}\}\le 9 \pl.\]
Now identify $P(j)$ and $Q(j)$ with their amplifications
$P(j)\ten I$ and $Q(j)\ten I$ on $H_1\ten H_2$. Denote that $u(\theta,t)=e^{iP(\theta)t}$ and $v(\theta,t)=e^{iQ(\theta)t}$.
Then
$\{(P(j),Q(j)), j=-1,0,1\}$ defines the map on $\Gamma_0$ satisfying condition i) and ii) for constant $C'=2500$, and \\

iii) \emph{the commutant of $\{u(\theta, t), v(\theta, t)| \pl \theta \in \Gamma_n\}$ is properly infinite.}\\

\noindent Now assume that the maps $P, Q$ are defined on $\Gamma_n$ with conditions i), ii) and iii) satisfied. For the induction step, we first add two integer points $\theta=\pm (n+2)$, and then apply the Lemma \ref{interpolates} to the subintervals $[\frac{j}{k^n}, \frac{j+1}{k^n}]$ (note that at $j=0$, $(P(0),Q(0))$ is the left regular representation) and unit intervals ${[-(n+2), -(n+1)]}$, ${[n+1, n+2]}$. In particular, for the two unit intervals, we can apply the Lemma \ref{interpolates} of $k$-division $n+1$ times. Thus we extend the maps $P,Q$ to $\Gamma_{n+1}$ with i), ii) and iii) still satisfied. Indeed, for ii),
\begin{align*}\max\{\norm{P(\frac{j}{k^{n+1}})&-P(\frac{j+1}{k^{n+1}})}{},\norm{Q(\frac{j}{k^{n+1}})-Q(\frac{j+1}{k^{n+1}})}{}\}\\&\le 1224k^{-\frac{n+1}{2}}+45\frac{2500}{k} k^{-\frac{n}{2}}=(1224+2500\frac{45}{\sqrt{k}})k^{-\frac{n+1}{2}}\le 2500k^{-\frac{n+1}{2}}\pl,\end{align*}
where the last inequality follows from that $\sqrt{k}=90$. Thus by induction, we construct the maps $P,Q$ on $\Gamma$.

Finally, we extend $P,Q$ from the dense subset $\Gamma$ to $\R$. By Lemma \ref{eq}, the one-parameter unitary groups $u(\theta ,t)=e^{iP(\theta)t}, v(\theta, t)=e^{iQ(\theta)t}$ satisfy that
\[\max\{\norm{u(\theta ,t)-u(\theta' ,t)}{},\norm{v(\theta ,t)-v(\theta' ,t)}{}\}\le 2500|t|\cdot |\theta-\theta'|^{\frac12}\pl, \pl\pl t\in \R, \theta \in \Gamma \pl,\]
For a fixed $t$, we can continuously extend $u(\cdot, t),v(\cdot, t)$ for all $\theta\in \R$.
It can be proved by the same argument of Theorem 5.4 in \cite{HR} that this extension is again Lip$^{\frac12}$ continuous, but with a larger constant $C=320\cdot 2500 |t|= 800,000 |t|$. Namely, for all $t, \theta, \theta' \in \R$, our extension satisfies
\begin{align}
\label{lip} \max \{\norm{u(\theta ,t)-u(\theta' ,t)}{},\norm{v(\theta ,t)-v(\theta' ,t)}{}\}\le 800,000|t|\cdot |\theta-\theta'|^{\frac12}\pl,
\end{align}
The continuity implies that for each $\theta$, $u(\theta ,t)$ and $v(\theta ,t)$ are strongly continuous one-parameter unitary groups such that
\begin{align}u(\theta ,s)v(\theta ,t)=e^{ist\theta }v(\theta ,t)v(\theta ,s) \pl, \pl \forall  \pl s,t \in \R \pl.
\label{thetai} \end{align}
Finally we choose the self-adjoint operators $P(\theta)$ and $Q(\theta)$ as the infinitesimal generators of $u(\theta ,t)$ and $v(\theta ,t)$. By Lemma
\ref{eq}, ii) is satisfied with constant $C=800,000$.
\end{proof}

\begin{rem}{\rm Proposition 3.9 of \cite{HR} proves that if $(P,Q)$ is a representation of the Heisenberg relation of finite multiplicity on a Hilbert space $H$, there exists no strongly commuting pair $(P_0, Q_0)$ on $H$ such that $P-P_0$ and $Q-Q_0$ are bounded. The argument works for all $\theta\neq 0$, which implies that the $(P(\theta), Q(\theta))$ constructed above of a nonzero $\theta$ is a representation of infinite multiplicity. Also, at $\theta=0$ $(P(0),Q(0))$ is the regular representation of $\R^2$ of infinite multiplicity by construction. \label{theta0}}\end{rem}

The above theorem can be reformulated with one-parameter unitary groups. %Recall that a pair of unitaries $u, v$ satisfying the $uv=\exp (2\pi \theta i)vu$ is a canonical pair of generators of $A_\theta$ if the canonical map from $A_\theta$ to $C^*(u,v)$ is isomorphic. For irrational $\theta$, $A_\theta$ is simple \cite{} and hence every pair satisfy the commutation relation is a conanical pair. For all $\theta$, a pair is canoncal if and only if there exsts a tracial state $\tau$ on $C^*(u,v)$ so that
%\begin{align}\tau(\sum_{n,m}^{finite} a_{n,m} u^nv^m)=a_{0,0} \pl,\pl a_{n,m}\in \mathbb{C} \pl, \label{2trace}\end{align}
%for all finite ploynomial of $u,v$. (Lemma 4.3, \cite{HR}).
\begin{cor}\label{continuousd2}
Let $H$ be an infinite dimensional Hilbert space. There exist two maps $u,v: \R\times \R \to U(H)$ and a universal constant $C>0$ such that
\begin{enumerate}
\item[i)]for each $\theta \in\R$, $u(\theta, \cdot)$ and $v(\theta,\cdot)$ are strongly (strong operator topoloy) continuous one-parameter unitary groups satisfying
   \begin{align}\label{thetagroup}u(\theta,s)v(\theta,t)=e^{ist\theta}v(\theta,t)u(\theta,s)\pl, \pl s,t \in \R\pl,\end{align}
%\item[ii)] $u(\theta,s), v(\theta,t)$ are canonical generators of rotation algebra $A_{\theta st/2\pi}$.
\item[ii)]for each $t\in \R$, $u(\cdot, t)$ and $v(\cdot, t)$ are Lip$^{\frac{1}{2}}$ continuous,
\[\max\{\norm{u(\theta,t)-u(\theta',t)}{} and \norm{v(\theta,t)-v(\theta',t)}{}\}\le C|t||\theta-\theta'|^{\frac12} \pl,  \pl\forall \pl\theta, \theta' \in \R.\]
\end{enumerate}Moreover, for all $\theta$, $u(\theta, \cdot)$ and $ v(\theta,\cdot)$ are representations of \eqref{thetagroup} of infinite multiplicity.
\end{cor}

\section{Perturbations of Noncommutative Euclidean Space}
We now consider the case of dimension $d>2$.
From this section on, ${\theta=(\theta_{jk})_{j,k=1}^d}$ denotes a real skew-symmetric $d \times d$-matrix. Let $(P_1,P_2,\cdots,P_d)$ be a $d$-tuple of self-adjoint operators on a Hilbert space $H$. We say $(P_1,P_2,\cdots,P_d)$ satisfy the \emph{Heisenberg relations with parameter $\theta$}
\[ [P_j,P_k]=-\theta_{jk}I\pl ,\pl j,k=1,2,\cdots,d\pl \]
if the one-parameter unitary groups $u_1(s)=e^{iP_1s},u_2(s)=e^{iP_2s},\cdots ,u_d(s)=e^{iP_ds}$ satisfy
\begin{align}u_j(s)u_k(t)=e^{i{st}\theta_{jk}} u_k(t)u_j(s)\pl,  \pl s,t \in \mathbb{R},\pl j,k=1, 2, \cdots, d \pl .
\label{qes1}\end{align}
When $\theta$ is the zero matrix, $(P_1,P_2, \cdots , P_j)$ gives a unitary representation of $\R^d$. The left regular representation of $\R^d$ is the translation action on $L_2(\R^d)$ as follows,
\[P_jf=-i\frac{\partial f}{\partial x_j} \pl ,\pl  (u_j(t) f)(x_1,x_2, \cdots, x_d)=f(x_1,x_2, \cdots, x_j+t, \cdots, x_d) \pl.\]
The standard noncommutative case is that $d=2n$ and $\theta=\left[ {\begin{array}{cc} 0 & I_n \\ -I_n & 0 \end{array} } \right]$, where $I_n$ is the $n$-dimensional identity matrix. This gives the conanical commutation relations (CCR) which consists of $n$ pairs of Heisenberg relations that mutually commute, i.e.
\begin{align} [P_{j}, P_{j+n}]=-iI\pl, \pl \forall \pl 1\le j\le n\pl\pl, \pl \text{otherwise}  \pl\pl [P_j, P_k]=0 \pl  .
\label{standard}\end{align}
The Stone-von Neumann theorem applies here (cf. \cite{Hall}, Theorem 14.8): any irreducible representation of \eqref{standard} is unitarily equivalent to $n$-dimensional quantum mechanics model on $L_2(\R^n)$,
\begin{align}P_{j}f= -i\frac{\partial f}{\partial x_j}\pl , \pl P_{j+n}f (x_1,\cdots,x_n)= x_jf (x_1,\cdots,x_n) \pl , \pl j=1, \cdots, n\pl; \label{mp}\end{align}
any representation of \eqref{standard} is unitarily equivalent to a finite or infinite multiple of the irreducible representation. It is known that similar property holds for all nonsingular $\theta$, which we briefly discuss in the following.

%In this situation, the von Neumann algebra generated is $R^d_\theta =B(L_2(\R^{n}))$.%
We use boldface letters for real vectors such as $\bs=(s_1,s_2,\cdots,s_d)$. Given a $d$-tuple $(u_1,u_2,\cdots, u_d)$ of one-parameter unitary groups satisfying \eqref{qes1}, we introduce the following strongly continuous map
\[u: \R^d \to U(H)\pll,\pll u(\bs)=\exp(-\frac{i}{2}\sum_{j<k}\theta_{jk}s_js_k)u_1(s_1)u_2(s_2)\cdots u_d(s_d) \pl.\]
It satisfies the commutation relation
\begin{align}\label{proj}u(\bs)u(\bt)=e^{\frac{i}{2}\theta(\bs,\bt)}u(\bs+\bt)=e^{i\theta(\bs,\bt)}u(\bt)u(\bs) \pl, \pl \bs,\bt \in \R^d \pl,\end{align}
where $\theta(\bs,\bt)=\sum_{jk}\theta_{jk}s_jt_k$ is the symplectic bilinear form associated with $\theta$. \eqref{proj} is called a projective unitary representation of $\R^d$ (see Appendix for the definition), and
it is an equivalent formulation of \eqref{qes1}. Let $T=(T_{jk})^{d}_{j,k=1}$ be a real invertible matrix and $T^t$ be its transpose. Then $\tilde{\theta}=T\theta T^{t}$ is also a skew-symmetric real matrix, and its associated Heisenberg relation admits the following representation,
\[u(T{\bf s})u(T{\bf t})=e^{i\theta(T\bs,T\bt)}u(T{\bf t})u(T{\bf s})=e^{\tilde{\theta}({\bf s},{\bf t})}u(T{\bf t})u(T{\bf s}) \pl.\]
Since $T$ is invertible, the Heisenberg relations associated with $\theta$ and associated with $\tilde{\theta}$ generate each other and there is an one-to-one correspondence bewteen their representations. Let us fix $S=\left[ {\begin{array}{cc} 0 & I_n \\ -I_n & 0 \end{array} } \right]$. For a nonsingular $\theta$, there exists an invertible $T$ such that $T\theta T^{t}=S$. Hence the Stone-von Neumann theorem concludes the case for all nonsingular $\theta$.

\begin{prop}\label{ns} Let $d=2n$. Suppose $\theta$ is nonsingular and $T=(T_{jk})^{d}_{j,k=1}$ is a real invertible matrix such that $T\theta T^t=S$. %Then the universal von-Neumann algebra of $\theta$ relation is $B(L_2(\mathbb{R}^k))$.
Then any irreducible representation of
\begin{align}[P_j,P_k]=-i\theta_{jk}I \pl, \pl j,k =1,2,\cdots,d \label{re}\end{align}
is unitarily equivalent to the following representation on $L_2(\mathbb{R}^{n})$,
\begin{align} \label{nonsingular}
P_j=\sum_{1\le k\le n} (T_{j,k} (-i\pd{x_k})+T_{j,k+n} x_k)\pl, \pl\pl\pl\pl j=1,2,\cdots,  d\pl.
\end{align} Moreover, any representation of \eqref{re} is a (finite or infinite) multiple of \eqref{nonsingular}.
\end{prop}

In general, for a singular $\theta$ the representation always generates a tensor product of Type I factor and commutative algebra (see \cite{Rieffel93}).
The next theorem is the generalization of Theorem \ref{bounded} for dimension $d\ge 2$.
\begin{theorem}\label{boundedd} Let $\theta$ be nonsingular. Let $(P_1, P_2,\cdots, P_d)$ be a representation of
 \[[P_j,P_k]=-i\theta_{jk}I\pl\pl, \pl\pl j,k=1,2,\cdots, d\pl\]
 on a separable Hilbert space $H$ of infinite multiplicity. Then for any real skew-symmetric $\theta'$, there exist self-adjoint operators $P'_1, P'_2,\cdots, P'_d$ on $H$ such that for all $j,k$,
\begin{enumerate}
\item[i)] $[P'_j, P'_k]=-i\theta_{jk}'I$;
\item[ii)]$P_j-P_j'$ is bounded on $H$ and \[\norm{P_j-P'_j}{}<9(d-1)\max_{ k } |\theta_{jk}-\theta_{jk}'|^{\frac{1}{2}}\pl.\]
\end{enumerate}
\end{theorem}
\begin{proof}Let $K$ be a separable infinite dimensional Hilbert space. Set $\displaystyle \delta_j=\max_{k}  |\theta_{jk}-\theta_{jk}'|^{1/2}$. Let us first assume that $\delta_j>0$ for every $j$. By Proposition \ref{ns}, up to a unitary equivalence we may assume that $H=L_2(\R^{d}, K)$ and $(P_1,P_2,\cdots, P_d)$ are given by
\begin{align}\label{rp}P_j=(-i\delta_j\pd{x_j}+\sum_{k<j}\frac{\theta_{kj}}{\delta_k}x_k)\ten 1_K ,\end{align}
where $(-i\pd{x_j})$'s and $x_j$'s are given in \eqref{mp}.  Let $W: \mathbb{R}^d\to U(K)$ be a $C^1$-function with values in the unitary group $U(K)$. $W$ can be viewed as a unitary on $L_2(\R^{d}, K)$ via pointwise action. A calculation  similar  to \eqref{cal} yields
\[W^*P_jW=-i\delta_j(\pd{x_j}+W^*\frac{\partial W}{\partial x_j})+\sum_{ k<j}\frac{\theta_{kj}}{\delta_k}x_k
 \pl .\]
 %Similar to Theorem \ref{bounded}, it is sufficient to show that there exists a $C^1$-function $W_d$ such that
%\begin{align}\sup_{(x_1,x_2,\cdots,x_d)\in \R^d}\norm{\frac{\partial W}{\partial x_j}(x_1,x_2,\cdots,x_d)-\sum_{1\le l<j}(\theta_{lj}'-\theta_{lj})wx_l}{}< 9(d-j)+5\sum_{l<j} |\theta_{lj}| \pl, \pl \label{W}\end{align}
%for all $j=1,2, \cdots, d$.
Let us recall the two-variable $C^1$-function $w: \mathbb{R}^2 \to U(K)$  in Theorem \ref{bounded}. Write ${\delta_{jk}= \frac{\theta_{jk}'-\theta_{jk}}{\delta_j\delta_k}}$ and define the following functions,
\begin{align*}
w_2(x_1,x_2,\cdots,x_d)&=w(x_1, \delta_{12}x_2)\pl,\\
w_3(x_1,x_2,\cdots,x_d)&=w(x_1,\delta_{13}x_3)w(x_2,\delta_{23}x_3)\pl ,\\
&\vdots\\
w_d(x_1,x_2,\cdots,x_d)&=w(x_1,\delta_{1d}x_d)w(x_2,\delta_{2d}x_d)\cdots w(x_{d-1},\delta_{d-1,d}x_d)\pl.
\end{align*}
 When $1\le j< k$,  for any $(x_1,x_2,\cdots,x_d)\in \R^d$
\begin{align}\label{kj}
\norm{\frac{\partial w_k}{\partial x_j}(x_1,x_2,\cdots,x_d)}{}=\norm{\frac{\partial w}{\partial x}(x_j,\delta_{jk}x_k)}{}<9 \pl,
\end{align}
and when $k< j \le d$, $\frac{\partial w_k}{\partial x_j}=0$. Because the pointwise unitaries $w(x_k, \delta_{kj}x_j)$ commutes with the multipliers $x_{j}$, we have
\begin{align}\label{jj}
&\frac{\partial w_j}{\partial x_j}(x_1,x_2,\cdots,x_d)-\sum_{k<j}\delta_{kj}x_kw_j(x_1,x_2,\cdots,x_d) \nonumber\\
=&\sum_{k<j}\delta_{kj}w(x_1,\delta_{1j}x_j)\cdots\frac{\partial w}{\partial x}(x_k,\delta_{kj}x_j)\cdots w(x_{j-1},\delta_{j-1,j}x_j)-\sum_{k<j}\delta_{kj}x_kw_j(x_1,x_2,\cdots,x_d) \nonumber\\
=&\sum_{k<j} \delta_{kj}w(x_1,\delta_{1j}x_j)\cdots(\frac{\partial w}{\partial y}(x_k,\delta_{kj}x_j)-ix_kw(x_k,\delta_{kj}x_j))\cdots w(x_{j-1},\delta_{j-1,j}x_j)
\pl.
\end{align}
Thus the norm estimate follows
\begin{align}
\norm{\frac{\partial w_j}{\partial x_j}(x_1,x_2,\cdots,x_d)-\sum_{k<j}\delta_{kj}x_kw_j(x_1,x_2,\cdots,x_d)}{}\le 9\sum_{k<j} \delta_{kj}.
\end{align}
Now set \[W(x_1,x_2,\cdots x_d)=w_2(x_1,x_2,\cdots,x_d)w_3(x_1,x_2,\cdots,x_d)\cdots w_d(x_1,x_2,\cdots,x_d) \pl.\]
and define $\displaystyle P'_j=W^*(P_j +\sum_{1\le k<j}\frac{\theta_{kj}'-\theta_{kj}}{\delta_k}x_k)W$. Then $(P'_1, P'_2, \cdots ,P'_d)$ satisfies \\${[P'_j, P'_k]=-i\theta_{jk}I}$ and for each $j$
\begin{align*}
P_j-P'_j=i(\delta_j W^*\frac{\partial W}{\partial x_j}-i\sum_{1\le k<j}\frac{\theta_{kj}'-\theta_{kj}}{\delta_k}x_k)=i\delta_j(W^*\frac{\partial W}{\partial x_j}-i\sum_{1\le k<j} \delta_{kj} x)\pl.
\end{align*}
Note that for all $(x_1,x_2,\cdots,x_d)\in \R^d$
\begin{align*}
&\norm{\frac{\partial W}{\partial x_j}(x_1,x_2,\cdots,x_d)-\sum_{1\le k<j} \delta_{kj} x_kW(x_1,x_2,\cdots,x_d)}{}\\ \le &\norm{\frac{\partial w_j}{\partial x_j}(x_1,x_2,\cdots,x_d)-\sum_{1\le k<j} \delta_{kj} x_kw_j(x_1,x_2,\cdots,x_d)}{}+
\sum_{ j<k \le d}
\norm{\frac{\partial w_k}{\partial x_j}(x_1,x_2,\cdots,x_d)}{}\\ \le& 9\sum_{1\le k<j} \delta_{kj}+9(j-1)=9(j-1)+9(d-j)=9(d-1)\pl.
\end{align*}
Therefore, $P_j-P'_j$ is bounded on $H$ and
$\norm{P_j-P'_j}{}\le 9 (d-1)\delta_j$.

For a general $\theta'$, we may assume that $\delta_j>0$ for $j\le s$ and $\delta_j=0$ for  $s<j\le d$. Then we take the representation \eqref{rp} for $j\le s$ and use
\[P_j=(-i\pd{x_j}+\sum_{k\le s}\frac{\theta_{kj}}{\delta_k}x_k + \sum_{ s< k\le d}\theta_{jk}x_k)\ten 1_K,\]
for $s<j\le d$. Applying the above argument to $P_1, \cdots, P_s$, we obtain
$${\displaystyle P'_j=W^*(P_j +\sum_{ k<j}\frac{\theta_{kj}'-\theta_{kj}}{\delta_k}x_k)W \pl\pl , \pl\pl j \le s}\pl.$$
Note that now the pointwise unitary $W$ is independent of coordinates $x_{s+1}, \cdots, x_d$, and hence it commutes with the newly defined $P_{s+1}, \cdots, P_d$. One can verify that the $d$-tuple ${(P'_1, \cdots, P'_s, P_{s+1},\cdots, P_{d})}$ satisfies the desired conditions.
\end{proof}
The next proposition is a partial converse of above theorem. The proof is a natural generalization of the Proposition 3.7 in \cite{HR}.
\begin{prop}\label{converse}Let $\theta$ be nonsingular and $(P_1,P_2,\cdots,P_d)$  be a representation of  \[[P_j,P_k]=-i\theta_{jk}I\pl\pl, \pl\pl j,k=1,2,\cdots, d\pl\] on a Hilbert space $H$. If $(P_1,P_2,\cdots,P_d)$ is of finite multiplicity, then there exist no strongly commuting self-adjoint operators $(P_1',P_2',\cdots, P_d')$ on $H$ such that $P_j-P_j'$ is bounded on $H$ for all $j$.
\end{prop}
\begin{proof}Let us first assume that $(P_1,\cdots,P_d)$ is irreducible. It is sufficient to consider the standard representation on $L_2(\R^n)$,
\[P_{j}= -i\pd{x_j}\pl , \pl P_{j+n}= x_j \pl , \pl j=1, \cdots, n\pl.\]
Other $\theta$'s follow by a linear transformation $T$ as in Proposition \ref{nonsingular}. Consider the creation and annihilation operators of the $n$-dimensional harmonic oscillator,
\[a_j=\frac{1}{\sqrt2}(P_{j}-iP_{j+n})\pl , \pl a_j^*=\frac{1}{\sqrt2}(P_{j}+iP_{j+n}) \pl.\]
We use the usual notations $m=(m_1,m_2,\cdots, m_n) \in \mathbb{N}^n , {m!=m_1!\cdots m_n!}$ and ${|m|=\sum_j m_j}$.
Denote $\phi_0(x)=\pi^{-\frac{n}{4}}e^{-\frac{|x|^2}{2}}$ as the Gaussian function for $\R^n$. There is a natural orthonormal basis of harmonic oscillator \[\phi_m=(a^*)^m \phi_0\pl , \pl m\in \mathbb{N}^n\]
where $(a^*)^m=(a_1^*)^{m_1}\cdots (a_n^*)^{m_n}$. The creation and annihilation actions are
\begin{align*}&a_j^*\phi_m=\sqrt{m_j+1}\phi_{(m_1,\cdots,m_j+1,\cdots, m_n)}\pl ,\pl a_j\phi_m=\sqrt{m_j}\phi_{(m_1,\cdots,m_j-1,\cdots, m_n)}, \\ &a_j\phi_{(m_1,\cdots,0,\cdots, m_n)}=0 \pl.
\end{align*}
Let $c_1 ,\cdots, c_n$ be the $N\times N$ matrices which are self-adjoint generators of the complex Clifford algebra $\C l^n$ ($N=2^{\frac{n}{2}}$ or $N=2^{\frac{n+1}{2}}$). They satisfy the commutation relations
\[c_jc_k+c_kc_j=2\pl , \pl \text{if} \pl j=k \pl ;\pl  c_jc_k+c_kc_j=0\pl , \text{otherwise}.\]
Set $A=\sum_j c_j\ten a_j^*$. One calculates that
\[A^*A=1\ten \sum_j a_ja_j^*\pl , \pl AA^*=1\ten \sum_j a_j^*a_j \pl.\]
Note that \[(\sum_j a_ja_j^*) \phi_m=(|m|+n) \phi_m \pl , \pl (\sum_j a_j^*a_j)\phi_m=|m|\phi_m,\]
Thus $|A|=(A^*A)^{\frac{1}{2}}$ is invertible with compact inverse $|A|^{-1}$ and $\ker(|A^*|)=\{\C \phi_0\}\ten \C^{N}$. The polar $V=A|A|^{-1}$  of $A$ is a partial isometry, $\ker(V)=\ker(|A|)$, and $\ker{V^*}=\ker(|A^*|)$. Hence $V$ is a Fredholm operator with $index(V)=-N$.

Assume that $P_1',\cdots, P_d'$ on $H$ commute strongly and $P_j'-P_j$ is bounded for all $j$. Then $A'=\frac{1}{\sqrt{2}}\sum_{j=1}^n c_j\ten (P_{j}'+iP_{2j+n}')$ is normal, and $A'-A$ is bounded on $H$. Let ${V'=A'|A|^{-1}}$. $V'$ is everywhere defined and bounded since $V'-V=(A'-A)|A|^{-1}$ is compact.. Hence $V'$ is also a Fredholm operator with ${index(V')=index(V)=-N}$. Nevertheless, since $|A|^{-1}$ is one-to-one and onto the domain $D(A)(=D(A'))$,
\[dim(\ker (V'))=dim (\ker (A'))=dim (\ker (A')^*)=dim(\ker (V')^*).\]
which leads to a contradiction.

When $(P_1,\cdots,P_d)$ has finite multiplicity $M$, $V$ is of Fredholm index $-MN$. The proof remains the same as above.
\end{proof}
%By \eqref{kj} and \eqref{jj}, $P_j-P'_j$ is
%by the product rule, we have
%\begin{align*}
%\norm{\delta_j\frac{\partial W}{\partial x_j}(x_1,x_2,\cdots,x_d)-i\sum_{k<j}\theta_{kj}x_kW(x_1,x_2, cdots,x_d)}{}\le
%\norm{\delta_j\frac{\partial w_j}{\partial x_j}(x_1,x_2,\cdots,x_d)(x_1,x_2, cdots,x_d)-i\sum_{k<j}\theta_{kj}x_kw_j(x_1,x_2, cdots,x_d)}{}
%+\sum_{j<k\le d}
%\norm{\delta_j\frac{\partial w_k}{\partial x_j}(x_1,x_2,\cdots,x_d)(x_1,x_2, cdots,x_d)}{}\le 9\delta_{j}
%< 9\sum_{k<j} j+5\sum_{l<j} |\theta_{lj}|
%\end{align*}

We now prove our main Theorem \ref{ultimate}. Recall that we denote by $\A(d)$ the space of all real $d$-dimensional skew-symmetric matrix.
\begin{proof}[Proof of Theorem \ref{ultimate}]The proof is by induction. The continuous maps $u,v$ from Theorem \ref{d2} give the initial step $d=2$. For the induction step, we may write ${H=H_1\ten H_2^{\ten (d-1)}}$, where both $H_1$ and $H_2$ are infinite dimensional. We assume $d-1$ maps $U_1,U_2, \cdots, U_{d-1}$ on $H_1$ satisfying the desired poperty for dimension $d-1$. Also we have the two maps $u,v$ for $d=2$ on $H_2$. Denote $\hat{\theta}$ for $(d-1)\times (d-1)$ principal submatrix $(\theta_{jk})_{j,k=1}^{d-1}$, $I_1$ the identity on $H_1$ and $I_2$ the identity on $H_2$. We constructed $d$ maps from $\A(d)$ to $U(H_1\ten H_2^{\ten (d-1)})$ as follows,
\begin{align*}
u_1(\theta,t)=&U_1(\hat{\theta},t)\ten u(\theta_{1d},t)\ten I_2 \ten\cdots \ten I_2 \pl, \\
u_2(\theta,t)=&U_2(\hat{\theta},t)\ten I_2\ten u(\theta_{2d},t)\ten I_2\ten\cdots \ten I_2 \pl,\\
u_3(\theta,t)=&U_3(\hat{\theta},t)\ten I_2 \ten I _2\ten u(\theta_{3d},t) \ten I_2 \ten \cdots \ten I_2 \pl,\\
&\vdots\\
u_{(d-1)}(\theta,t&)=U_{(d-1)}(\hat{\theta},t)\ten I_2 \ten \cdots \ten I _2 \ten u(\theta_{(d-1),d},t) \pl,\\
u_d(\theta,t)=&I_{1}\ten v(\theta_{1d},t)\ten v(\theta_{2d},t) \ten v(\theta_{3d},t) \ten \cdots \ten v(\theta_{(d-1),d},t) \pl.
\end{align*}
One can check that  $(u_1,u_2,\cdots,u_d)$ satisfies the desired commutation relations. By the triangle inequality, for $j\le d-1$,
\begin{align}\norm{u_j(\theta, t)-&u_j(\theta', t)}{}\le \norm{U_j(\hat{\theta}, t)-U_j(\hat{\theta'}, t)}{}+\norm{u(\theta_{jd}, t)-u(\theta'_{jd}, t)}{}  \label{norm} \\&\le C|t|(\sum_{1\le k\le d-1}|\theta_{jk}-\theta'_{jk}|^{\frac{1}{2}})+C|t| \pl |\theta_{jd}-\theta'_{jd}|^{\frac{1}{2}}= C|t|(\sum_{1\le k\le d}|\theta_{jk}-\theta'_{jk}|^{\frac{1}{2}})\pl, \nonumber\end{align}
and the estimate of $u_d$ follows similarly. Here we used the inductive assumption on $d-1$ and the initial step on $d=2$.
The constant $C$ is independent of dimension $d$ and it can be $800,000$ as in Theorem \ref{continuousd2}.
\qd

Denote $|\bs|=(\sum_j |s_j|^2)^\frac{1}{2}$ as the Euclidean metric for vector $\bs=(s_1,s_2,\cdots,s_d)$ and recall the symplectic bilinear form $\theta(\bs,\bt)=\sum_{jk}\theta_{jk}s_jt_k$. A strongly continuous map $u:\R^d \to U(H)$ is called a $\theta$-projective unitary representation (or shortly $\theta$-representation) of $\R^d$ if it satisfies \begin{align}\label{po}u(\bs)u(\bt)=e^{\frac{i}{2}\theta(\bs,\bt)}u(\bs+\bt)\pl.\end{align}
The above theorem can be reformulated as a continuous family of projective unitary representations.

\begin{cor}
\label{projrep} Let $H$ be an infinite dimensional Hilbert space.
There exist a map $u: \A(d)\times \R^d  \to U(H)$ and a universal constant $C>0$ such that
\begin{enumerate}
\item[i)] for each $\theta\in\A(d)$, $u(\theta, \cdot)$ is a strongly continuous $\theta$-representation of $\R^d$;
\item[ii)]for any $\bs\in \R^d$ and $\theta, \theta'\in \A(d)$,
 \begin{align}\label{esti}\norm{u(\theta,\bs)-u(\theta',\bs)}{}\le C|\bs|(\sum_{k,j} |\theta_{kj}-\theta_{kj}'|^{\frac12})\pl, \end{align}
\end{enumerate}
\end{cor}
\begin{proof}Let $u_1(\theta,\cdot),u_2(\theta,\cdot),\cdots, u_d(\theta,\cdot)$ be one-parameter unitary groups from Theorem \ref{ultimate} and $P_1(\theta),P_2(\theta),\cdots,P_d(\theta)$ be the corresponding infinitesimal generators. Then
\[[P_j(\theta), P_k(\theta)]=-i\theta_{jk}I\pl , \pl j,k=1,\cdots, d,\]
and by Lemma \ref{eq}, $P_j(\theta)-P_j(\theta')$ is bounded on $H$,
\[\norm{P_j(\theta)-P_j(\theta')}{}\le C(\sum_{ k}|\theta_{jk}-\theta'_{jk}|^{\frac{1}{2}}) \pl.\]
Define
 \[u(\theta,{\bf s})=\exp(-\frac{i}{2}\sum_{j<k}\theta_{jk}s_js_k)u_1(\theta,s_1)u_2(\theta,s_2)\cdots u_d(\theta,s_d)\]
We first consider that $\theta$ and $\theta'$ are nonsingular. It is clear from Proposition \ref{ns} that $P_1(\theta), P_2(\theta), \cdots,P_d(\theta)$ share a common core. For any vector $\bs \in \R^d$, it follows from Baker-Campbell-Hausdorff formula that
 \[u(\theta,{\bf s})=\exp(-\frac{i}{2}\sum_{j<k}\theta_{jk}s_js_k)u_1(\theta,s_1)u_2(\theta,s_2)\cdots u_d(\theta,s_d)=e^{i(\sum_j s_jP_j(\theta))} \pl.\]
Therefore,
\begin{align*}\norm{u(\theta, {\bf s})-u({\theta'}, {\bf s})}{}&=\norm{e^{i(\sum_j s_j P_j(\theta))}-e^{i(\sum_j s_j P_j(\theta'))}}{}\le\norm{\sum_j s_j (P_j(\theta)- P_j(\theta'))}{}\\&\le C\sum_j |s_j|(\sum_{k}  |\theta_{jk}-\theta_{jk}'|^{\frac12})\le C|{\bf s}|(\sum_j (\sum_k |\theta_{jk}-\theta_{jk}'|^{\frac12})^2)^\frac{1}{2}\\&\le C|{\bf s}|(\sum_{j,k} |\theta_{jk}-\theta_{jk}'|^{\frac12})\pl.
\end{align*}
Note that when $\bs$ is fixed, $u(\cdot,\bs)$ is continuous in norm. Then the estimates for general $\theta$ and $\theta'$ follows.
\end{proof}

We now explore applications on noncommutative Euclidean spaces. Let $\S(\R^d)$ be the complex Schwartz functions on $\R^d$, and $\bs\cdot \bt$ denote the Euclidean inner product. Fix a nonzero $\theta\in \A(d)$.  The associated Moyal product, for $f,g\in \S(\R^d)$, is defined as
\begin{align} f\star_\theta g (\bx)=(2\pi)^{-d}\int_{\R^{d}} \int_{\R^{d}} f(\bx-\frac{1}{2}\theta\bs)g(\bx+\bt)e^{-i\bs\cdot \bt}d\bs d\bt \pl, \pl f,g\in \S(\R^d) \pl
\label{moyalformula}\end{align}
(see \cite{Rieffel93}). The noncommutative Euclidean space $E_\theta$ associated to $\theta$ is the $C^*$-algebra generated by $\{\la_\theta(f)\pl|\pl f\in \S(\R^d)\}$, where $\la_\theta: (S(\R^d),\star_\theta)\to B(L_2(\R^d))$ is the left Moyal multiplication,
\[\la_\theta(f)g=f\star_\theta g \pl, \pl f\in \S(\R^d)\pl ,\pl g\in L_2(\R^d) \pl.\]
$E_\theta$ is a noncommutative deformation of $C_0(\R^d)$, the algebra of continuous functions vanishing at infinity. An equivalent formulation of Moyal product via Fourier transform is given by
\begin{align*}
\la_\theta(f)= \int_{\R^{d}} \hat{f}(\bs) \la_\theta(\bs) d\bs \pl,
\end{align*}
where $\hat{f}$ is the Fourier transform of $f$ and
\begin{align}\label{regulartheta}\widehat{\la_\theta(\bs)g}(\bt) =e^{i\theta(\bs,\bt-\bs)}\hat{g}(\bt-\bs) \pl,\end{align}
is called the left regular $\theta$-representation of $\R^d$ (see \cite{Gracia88} for more information on Moyal analysis).

Let $u$ be a $\theta$-representation on a Hilbert space $H$. The associated quantization map (still denoted by $u$)
\[ \pl u(f)=\int_{\R^{d}} \hat{f}(\bs) u(\bs) d\bs \pl,\]
gives a representation of Moyal product, i.e. $u(f)u(g)=u(f\star_\theta g)$.
Actually, representations of $E_\theta$ are in one-to-one correspondence to $\theta$-representations of $\R^d$ (see \cite{Mackey} and Appendix).
There is a (canonical) $*$-homomorphism from $E_\theta$ onto the $C^*$-algebra generated by $\{u(f)\pl|\pl f\in \S(\R^d)\}$,
\[\pi_u(\la_\theta(f))= u(f)\pl.\]
If $\pi_u$ is an isomorphism, we say $u$ is \emph{canonical}.

\begin{lemma}\label{reduced} Let $\theta\in \A(d)$ be nonsingular and $\hat{\theta}=(\theta_{jk})_{j,k=1}^{d-1}$ be its principal submatrix. For a vector $\bs\in \R^{d-1}$, write $(\bs,0)=(s_1,\cdots,s_{(d-1)},0)\in \R^d$. Let $u :\R^d \to U(H)$ be a $\theta$-representation on $H$.
Then the following $\hat{\theta}$-representation of $\R^{d-1}$
 \[\hat{u}: \R^{d-1} \to U(H)\pl ,\pl \hat{u}(\bs)=u(\bs,0) \pl, \]
is canonical.
\end{lemma}
\begin{proof} For $f\in \S(\R^{d-1})$, denote $v(f)=\int \hat{f}(\bs)\la_\theta(\bs,0)d\bs$, and $\la_{\hat{\theta}}(f)$ for the left $\hat{\theta}$-Moyal multiplication on $L_2(\R^{d-1})$. It is sufficient to show that
\[\norm{v(f)}{}\pl = \pl\norm{\la_{\hat{\theta}}(f)}{} \]
holds for functions $f$ which are $L_1$-norm dense in $\S(\R^{d-1})$.
By Proposition \ref{nonsingular}, we may just consider that $u$ is the left regular $\theta$-representation \eqref{regulartheta} on $L_2(\R^d)$. For any $g\in S(\R^{d-1})$, define $g_n\in S(\R^{d})$ as follows,
\[\hat{g}(\bt,t_d)=\hat{g}(\bt)\phi_n(t_d) \pl ,\pl (\bt,t_d)\in \R^d \pl,\]
where $\phi_n\in S(\R)$ is a sequence of smooth function supported in $[-\epsilon_n, \epsilon_n]$  such that $\epsilon_n \to 0$ and the $L_2$-norm $\norm{\phi_n}{2}=1$.
For $f\in \S(\R^{d-1})$, \begin{align*}&\widehat{v(f)g_n}(\bt,t_d)-\widehat{\la_{\hat{\theta}}(f)g}(\bt) \phi_n(t_d)\\=&\phi_n(t_d)\int \hat{f}(\bs)\hat{g}(\bt-\bs)e^{\frac{i}{2}\hat{\theta}(\bs,\bt-\bs)} \exp(\frac{i}{2}\sum_{j=1}^{d-1}\theta_{jd} s_jt_d)d\bs-\phi_n(t_d)\int \hat{f}(\bs)\hat{g}(\bt-\bs)e^{\frac{i}{2}\hat{\theta}(\bs,\bt-\bs)} d\bs
\\ =& \phi_n(t_d)\int \hat{f}(\bs)\hat{g}(\bt-\bs)e^{\frac{i}{2}\hat{\theta}(\bs,\bt-\bs)} (\exp(\frac{i}{2}\sum_{j=1}^{d-1}\theta_{jd} s_jt_d)-1)d\bs\pl.
\end{align*}
Now assume that $\hat{f}$ is compactly supported. Then the sequence
\[ \beta_n:=\sup_{t_d\in supp (\phi_n)}\sup_{\bs \in supp (f)}|\exp(\frac{i}{2}\sum_{j=1}^{d-1}\theta_{jd} s_jt_d)-1|\]
converges to $0$. Hence
\begin{align*}
\norm{\widehat{v(f)g_n}-\widehat{\la_{\hat{\theta}}(f)g}\phi_n}{2}\le
\beta_n\norm{\phi_n}{2}\norm{f}{2}\norm{g}{2} \to 0
\end{align*}
Thus for compactly supported $f$, and any $g\in \S(\R^{d-1})$
\[\lim_n\norm{v(f)g_n}{2}\pl= \pl \norm{\la_{\hat{\theta}}(f)g}{2} \pl,\]
which implies
\[\norm{v(f)}{}\pl = \pl\norm{\la_{\hat{\theta}}(f)}{} \pl.\]
because $\norm{g}{2}=\norm{g_n}{2}=1$.
\end{proof}
Let $\al>0$ and $\Delta$ be the Laplacian in $\R^d$. Recall that the Sobolev space $W^{\al, 2}(\R^d)$ is the Hilbert space ${H^\al(\R^d)=\{f\in L_2(\R^d) \pl| \pl (1+|\Delta|)^{\frac{\al}{2}}f \in L_2(\R^d)\}}$ equipped with the norm $\norm{f}{H^\al}=\norm{(1+|\Delta|)^{\frac{\al}{2}}f}{2}$.
\begin{cor} \label{quan}Let $H$ be an infinite dimensional Hilbert space. There exists a map ${u:\A(d)\times \mathcal{S}(\R^d) \to B(H)}$, $(\theta,f) \to u_\theta(f)$ such that
\begin{enumerate}
\item[i)] for each $\theta \in \A(d)$,
\[u_\theta(f\star_\theta g)=u_\theta(f)u_\theta(g)\pl , \pl \forall \pl f,g\in \mathcal{S}(\R^d) \pl;\]
\item[ii)] for $\al>\frac{d}{2}+1$, there exists a constant $C_{\al,d}$ such that for all $f\in \mathcal{S}(\R^d)$ and ${\theta, \theta'\in \A(d)}$,
\[\norm{u_\theta(f)-u_{\theta'}(f)}{}\le C_{\al,d}(\sum_{j,k} |\theta_{jk}-\theta'_{jk}|^{\frac{1}{2}})\norm{f}{H^\al}\pl;\]
\item[iii)] for all $f$ and $\theta$,
$\norm{u_\theta(f)}{}= \norm{\la_\theta(f)}{}\pl.$
\end{enumerate}
\end{cor}
\begin{proof}We first consider the case $d=2m$ is even. Let $u(\theta,\bs)$ be the continuous family of projective unitary representations from Corollary \ref{projrep}. Define
\[u_\theta(f)=\int_{\R^d}\hat{f}({\bf s})u(\theta,{\bf s})d{\bf s}\pl.\]
The first assertion follows from that $u(\theta, \cdot)$ is a $\theta$-representation of $\R^d$. For ii), we use the estimate \eqref{esti},
\begin{align*}\norm{u_\theta(f)-u_{\theta'}(f)}{}&\le C (\sum_{j,k} |\theta_{jk}-\theta_{jk}'|^{\frac12})\int_{\R^d} |\hat{f}({\bf s})||{\bf s}|d{\bf s}\\
&\le C (\sum_{j,k} |\theta_{jk}-\theta_{jk}'|^{\frac12})\norm{|\hat{f}||{\bf s}| (1+4\pi|{\bf s}|^2)^{\frac{\al}{2}-\frac{1}{2}}}{2}(\int_{\R^d} (1+4\pi |{\bf s}|^2)^{-\al+1} d {\bf s})^{\frac12} \\
&\le C_{\al, d} (\sum_{j,k} |\theta_{jk}-\theta_{jk}'|^{\frac12})\norm{f}{H^\al}\pl,
\end{align*}
The second integral converges when $\al-1>\frac{d}{2}$ and the constant \[C_{\al, d}= (\frac{V_d}{2\al-d-2})^{\frac12}C \]
where $V_d$ is the volume of unit $(d-1)$-sphere. For iii), given a $\theta$, it is sufficient to show that for any $f\in \mathcal{S}(R^d)$, $\norm{u_\theta(f)}{}=\norm{\la_\theta(f)}{}$.
This is clear for all nonsingular $\theta$. Given a singular $\theta$ in even dimensions, we choose a sequence of nonsingular skew-symmetric $\theta_n$ converging to $\theta$. With the continuity in ii), we obtain that for all $f\in H^\al (\R^d)$,
\[ \norm{u_{\theta}(f)}{}=\lim_n\norm{u_{\theta_n}(f)}{}=\lim_n\norm{\la_{\theta_n}(f)}{}\ge\norm{\la_{\theta}(f) }{}\pl.\]
The last inequality follows from that $\la_{\theta_n}(f)\to \la_\theta(f)$ in strong operator topology, and it is actually an equality (see \cite{Rieffel93}). Since $H^\al(\R^d)$ is $L_1$-norm dense in $\S(\R^d)$, we finish the proof for the even case. When $d=2m-1$ is odd,
we set \[u_\theta(f)= \int_{\R^{d-1}} \hat{f}(\bs)u(\tilde{\theta},(\bs,0))d\bs\pl, \pl f\in \S(\R^{d-1})\pl, \theta\in \A(d-1)\]
where $\tilde{\theta}=\left[\begin{array}{cc}\theta & 0 \\ 0 & 0\end{array}\right]$ is an embedding of $\A(d-1)$ into $\A(d)$. i) and ii) follows similarly. For iii), again we choose a sequence of nonsingular $\tilde{\theta}_n$ approximating $\tilde{\theta}$. Denote $\theta_n$ for corresponding the $(d-1)\times (d-1)$ principal submatrix of $\tilde{\theta}_n$. Then $\theta_n$ converges to $\theta$, and by Lemma \ref{reduced} we obtain that for all $f\in H^\al(\R^{d-1})$,
\[\norm{u_\theta(f)}{}=\lim_n\norm{\int \hat{f}(\bs)u({\tilde{\theta}_n}, (\bs,0))d\bs}{}=\lim_n\norm{\la_{\theta_n}(f)}{}\ge\norm{\la_\theta(f)}{} \pl.\]
We finish the proof.
\end{proof}
\begin{rem}{\rm All the projective unitary representations in Corollary \ref{projrep} can be canonical. For each $\theta$ the $C^*$-algebra generated by the quantization
\[u_\theta(f)=\int_{\R^d}\hat{f}({\bf s})u(\theta,{\bf s})d{\bf s}\pl, \pl f\in \S(\R^d)\]
is isomorphic to $E_\theta$.
}
\end{rem}

\section{The continuous maps of Noncommutative Tori}
Let $(u_1, u_2, \cdots, u_d)$ be a $d$-tuple of unitaries satisfy
\begin{align}u_ju_k=\sigma_{jk}u_ku_j\pl, \pl j ,k=1, 2,\cdots ,d,  \pl \label{atheta}
\end{align}
where $\sigma_{jk}=e^{2\pi i \theta_{jk}}$.
 We say $(u_1, u_2, \cdots, u_d)$ is a canonical $d$-tuple of generators for $A_\theta$ if the canonical map from $A_\theta$ to $C^*(u_1,u_2,\cdots, u_d)$ is a $*$-isomorphism. We denote $\T(d)\cong \T^{\frac{(d-1)d}{2}}$ as the space of all Hermitian $d\times d$ matrices with unit entries. In this section $u$ will denote a $d$-tuple of unitaries $(u_1,u_2, \cdots, u_d)$ and $m$ denote a $d$-tuple of integers $(m_1,m_2,\cdots, m_d)$. We use the standard notation of multiple Fourier series as follows,
\[ u^m=u_1^{m_1}u_2^{m_2}\cdots u_d^{m_d}\pl.\]
 A polynomial in $u$ with finite nonzero coefficients is
\[a= \sum_{m\in \Z^d} \al_mu^m \pl. \]
Denote $\mathcal{P}_\theta$ the $*$-algebra of all polynomials of $(u_1, u_2, \cdots, u_d)$. $A_\theta$ is the enveloping $C^*$-algebra of $\mathcal{P}_\theta$. One can define a faithful tracial state $\tau$ on $\mathcal{P}_\theta$,
 \[\tau(\sum_{m\in \Z^d}\al_m u^m)=\al_0 \pl.\]
The GNS-representation of $\tau$ is given as follows,
\begin{align}\pi(u^m)\ket{m'}=\exp({2\pi i (-\sum_{1\le j< k\le d} \theta_{jk}m_km^{'}_j)})\ket{m+m'}\pl, \pl \forall\pl m,m'\in \Z^d\pl,
\label{GNS}\end{align}
where we use ``kets'' $\ket{m}$ for the GNS-vector of $u^m$. $\{\ket{m}|m\in \Z^d\}$ forms an orthonormal basis and the Hilbert space is isomorphic to $l_2(\mathbb{Z}^d)$. The trace $\tau$ is implemented by the cyclic vector $\ket{\bf{0}}$,
\[\tau(\sum_{m\in \Z^d}\al_m u^m)=\bra{\bf{0}}\sum_{m\in \Z^d}\al_m \pi(u^m)\ket{\bf{0}}=\al_0\pl.\]
By universality, $\pi$ extends to a $*$-representation of $A_\theta$ and so does the tracial state $\tau$. To see that both $\tau$ and $\pi$ are faithful, we recall the following reformulation of $\tau$ by the transference automorphisms of $A_\theta$. Let $\T^d$ be the $d$-torus
\[\T^d=\{(z_1,z_2,\cdots, z_d) \in \C^d \pl |\pl |z_j|=1 \pl , \pl \forall  j\}\pl.\]
 For a $d$-tuple $z=(z_1,z_2,\cdots, z_d)\in \T^d$,
the transference automorphism associated to $z$ is given by
\[\al_z(u^m)=z^mu^m\equiv z_1^{m_1}z_2^{m_2}\cdots z_d^{m_d}u_1^{m_1}u_2^{m_2}\cdots u_d^{m_d} \pl,\]
and extended to $A_\theta$ by universality. For each $j$, we introduce the following map
\[\Phi_{j}(a)=\int_{\T} \al_{(1,\cdots, z_j, \cdots, 1)} (a) dz_j \pl.\]
As an averaging of automorphisms, $\Phi_{j}$ is faithful, completely positive and contractive. Note that
\[\Phi_{j}(u^m)=\left\{
\begin{aligned}
& u^m\pl  &  \pl \text{if}\pl m_j=0 \\
 & 0\pl  & \pl \text{otherwise} \\
\end{aligned}
\right. \pl .\]
$\Phi_{j}$ is the conditional expectation onto the subalgebra generated by all unitary generators except $u_j$. One can see that $\Phi_j\Phi_k=\Phi_k\Phi_j$, and this composition is the conditional expectation onto the subalgebra generated by all generators except for $u_j$ and $u_k$. Inductively, the map $\Phi_1\Phi_2\cdots \Phi_d$ is the conditional expectation onto the scalers, which coincides with the canonical state $\tau$,
\[\Phi_1\Phi_2\cdots \Phi_d(u^m)=\tau(u^m)I=\left\{
\begin{aligned}
&I &\pl \text{if}\pl \pl m =(0&,0,\cdots, 0) \\
&0 &\pl \text{otherwise}&
\end{aligned}
\right. \pl.\]
This justifies that $\tau$ is faithful and so is the representation $\pi$. %For an arbitrary element $a \in A_\theta$, one can defines its ``Fourier coefficients'' as
%\[a_m=\tau(a^*U^m) \pl.\]
%Hence the GNS representation of $\tau$ is faithful and gives a concrete form of $A_\theta$
%Another interesting prospective for $\tau$ and $\pi$ is by the iterated crossed product. For $d=2$, the rotation algebra can expressed by
%\[A_\theta=C^*(\Z)\rtimes_\theta \Z \pl.\]
%This can be generalized to $d$-dimensional as follows
%\[A_\theta=C^*(\Z)\rtimes_{\theta_{2}} \Z\rtimes_{\theta_{3}} \Z\cdots \rtimes_{\theta_d}\Z \pl,\]
%where $\theta_j, 2\le j\le d$ is an action detemined by the entry of $j$th row in $\theta$. The product trace from the trace on $C^*(\Z)$ coincides with $tau$. Moreover, since $\Z$ is amenable as group, the iterated crossed product is isomorphic to the reduced one, which is unitarily equivalent to GNS-construct above.

The following lemma is an analog of Lemma 4.3 in \cite{HR}.
\begin{lemma} \label{trace} Let $(u_1,u_2,\cdots, u_d)$ be a $d$-tuple of unitaries satisfying
\[u_ju_k=e^{2\pi i\theta_{jk}}u_ku_j\pl, \pl j ,k=1, 2,\cdots ,d \pl . \]
Then $(u_1,u_2,\cdots, u_d)$ is a canonical $d$-tuple of generators for $A_\theta$ if and only if there exists a state $\tau$ on $C^*(u_1,u_2,\cdots, u_d)$ such that,
\begin{align}\label{tau}\tau(u^m)=\left\{
\begin{aligned}
& 1 &\pl \text{if} \pl \pl m=(&0,0,\cdots, 0) \\
&0 &\pl \text{otherwise}&
\end{aligned}
\right. \pl.
\end{align}
\end{lemma}
\begin{proof} The necessity follows from the above discussion. Let us identify $A_\theta$ with the representation $\pi(A_\theta)\subset B(l_2(\Z^d))$ in \eqref{GNS}. Given a state $\tau$ as \eqref{tau},  the GNS-representation $\pi_\tau$ maps $C^*(u_1,u_2,\cdots, u_d)$ into $B(l_2(\Z^d))$ and sends each $u_j$ to the canonical unitary $\tilde{u_j}\in A_\theta$. Denote $\pi_u$ for the canonical map from $A_\theta$ onto $C^*(u_1,u_2,\cdots, u_d)$. Both compositions $\pi_u\pi_\tau$ and $\pi_\tau\pi_u$ are the identity maps, since they send generator unitaries to generator unitaries. Therefore the canonical $\pi_u$ is a $*$-isomorphism.
\end{proof}
The next theorem is a refinement of Theorem \ref{p} with periodicity.
 %We give a quick argument using the idea from \cite{Davidson} for $d=2$. Let us define
%\[S_j=\{\theta\in \Theta |(1, \theta_{j1}, \theta_{j2}, \cdots, \theta_{jd})\pl \text{is}\pl \mathbb{Q}
%\pl \text{independent}\}\pl\pl\text{and}\pl\pl S=\bigcap_{1\le j\le d} S_j.\]
%Note that for each $j$, $S_j$ is a dense subset of $\bar{\Theta}$ and so is $S$ by Baire Theorem.
%\begin{lemma}
%\label{simple}
%If $\theta\in S_j$, then
%\[\Phi_j(a)=\lim_{n\to \infty}\frac{1}{2n+1}\sum_{-n\le k\le n}U_j^{-k}a U_j^k \pl.\]
%In particular, if $\theta\in S$, then $A_\theta$ is simple.
%\end{lemma}
%\begin{proof}
%%Suppose $\theta\in S_j$, then $\sum_{k=1}^d\theta_{j,k}m_k$ is irrational for any nonzero $z\in \Z^d$. Then,
%\begin{align*}\Phi_j(u^m)=&\lim_{n\to \infty}\frac{1}{2n+1}\sum_{-n\le k\le n}u_j^{-k}u^m u_j^k=\lim_{n\to \infty}\frac{1}{2n+1}\sum_{-n\le j\le n}\exp(-im_j\theta_{jm})u^m\\=&\lim_{n\to \infty}\frac{1}{2n+1}\frac{\sin{(2n+1)\pi \sum_{k=1}^d\theta_{jm}m_j}}{sin{\pi \sum_{k=1}^d\theta_{jm}m_j}}u^m = \begin{cases}
%%u^m  & m_j=0 \\
%0  & \text{otherwise}
%%\end{align*}
%which verifies i). For ii), $\Phi_1\Phi_2\cdots\Phi_k=\tau$ is the faithful tracial state given in \eqref{tau}. Let $a$ be any nonzero element an closed ideal $I\subset A_\theta$. Then so is nonzero scalar $\tau(x^*x)$, because each $\Phi_j$ sends $I$ into $I$ faithfully.  Therefore $I$ is a trivial idea.
%\end{proof}
\begin{theorem}
Let $H$ be an infinite dimensional Hilbert space. There exist $d$ continuous maps $u_1,u_2, \cdots, u_d: \T(d) \to U(H)$ and a universal constant $C>0$ such that \begin{enumerate}
 \item[i)]for $\si$ such that $\si_{jk}=e^{2\pi i\theta_{jk}}$, $(u_1(\si),u_2(\si), \cdots, u_d(\si))$ is a canonical $d$-tuple of generators for $A_\theta$;
 \item[ii)] for each $j$,
\[ \frac12\max_{k}|\si_{jk}-\si'_{jk}|^{\frac{1}{2}}\le \norm{u_j(\theta)-u_j(\theta')}{}\le C(\sum_{ k }|\si_{jk}-\si'_{jk}|^{\frac{1}{2}})\pl. \]
\end{enumerate}
\end{theorem}
\begin{proof} The continuous maps and the upper estimates of ii) can be proved with same construction as in Theorem \ref{ultimate}. The lower estimates follows from Proposition 4.6 in \cite{HR}, for each pair of indices $(j,k)$. To show that $(u_1(\si),u_2(\si), \cdots, u_d(\si))$ is canonical, we recall the fact that $A_\theta$ is simple when $\theta \mathbb{Z}^d \cap \mathbb{Z}^d=\{\textbf{0}\}$ (see \cite{Sl, OTP, Gr}). Such $\theta$'s are dense in all skew-symmetric $d\times d$ matrices. Then the conclusion can be derived by combining the argument of Remark 5.6 in \cite{HR} with Lemma \ref{trace}.
\end{proof}
\begin{rem}{\rm For $\al>0$, let us recall the Soblev space on $d$-torus \[H_\al(\T^d)=\{f \in L_2(\T^d)\pl |\pl \sum_{m\in \Z^d} (1+|m|^2)^\al|\hat{f}(m)|^2<\infty\}\pl, \pl \norm{f}{H^\al}=(\sum_{m\in \Z^d} (1+|m|^2)^\al|\hat{f}(m)|^2)^{\frac12}\]
where $\hat{f}$ is the Fourier series of $f$. Given the $d$ continuous maps
$u_1,u_2, \cdots, u_d$ above, we have the following quantization,
\[u_\si(f):=\sum_m \hat{f}(m_1,m_2,\cdots,m_d)u_1(\si)^{m_1}u_2(\si)^{m_2}\cdots u_d(\si)^{m_d}, \pl \si\in \T(d) \pl. \]
The series on the r.h.s is well defined if $f\in H^\al(\T^d)$ for some $\al>\frac{d}{2}$. We have an analog of Corollary \ref{quan} as follows: for $\al>\frac{d}{2}+1$ there exists constant $C_{\al,d}$ depending on $\al,d$ such that
\[\norm{u_\si(f)-u_{\si'}(f)}{}\le C_{\al,d}\norm{f}{H^\al}\sum_{j,k}|\si_{jk}-\si'_{jk}|^{\frac12}\pl\]
holds for any
$f\in H^\al(\T^d)$ and $\si,\si'\in \T(d)$.}
\end{rem}
Let us define that for each pair $\si, \si' \in \T(d)$,
\[\rho(\si, \si')\equiv \inf \max_j \norm{u_j -u_j'}{} \pl,\]
where the infimum runs over all $d$-tuple of unitaries $(u_1,u_2, \cdots, u_d)$ on the seperable infinite dimensional Hilbert space $H$ satisfying the commutation relation \eqref{atheta} for $\si$, and respectively $(u'_1,u'_2, \cdots, u'_d)$ for $\si'$. It is proved in \cite{HR} that for $d=2$, $\rho$ is a translation invariant metric on $\T$ and
\[\frac12|\si-\si'|^{\frac12}\le \rho(\si, \si')\le 24 |\si-\si'|^\frac{1}{2},\pl \si, \si' \in \T \pl .\]
Their argument generalizes to $d> 2$.
\begin{prop} $\rho$ is a translation-invariant metric on $\T(d)$ and for any $\si, \si'\in \T(d)$,
\begin{align}\frac{1}{2}\max_{j,k}|\si_{jk}-\si_{jk}'|^{\frac12}\le \rho(\si, \si') \le 24(d-1)\max_{j,k} |\si_{jk}-\si_{jk}'|^{\frac12} \pl .\label{final}\end{align}
\end{prop}
\begin{proof} We first show the translation-invariance. Given $\si,\si',\si''\in \T(d)$, let $(u_1, u_2,\cdots u_d)$, $(u_1', u_2',\cdots u_d')$ and $(u_1'', u_2'',\cdots u_d'')$ be $d$ tuples of unitaries on $H$ satisfying
\[u_ju_k=\si_{jk}u_ku_j \pl, \pl u_j'u_k'=\si_{jk}'u_k'u_j' \pl, \pl u_j''u_k''=\si_{jk}''u_k''u_j''\pl, \pl j,k=1,\cdots, d.\]
Define the new unitaries
\[v_j=u_j\ten u_j'' \pl,  \pl v_j'=u_j\ten u_j'', \pl ,j=1,\cdots,k\]
on $H\ten_2 H\cong H$, they satisfy
\[v_jv_k=\si_{j,k}\si''_{j,k}v_kv_j \pl , \pl v_j'v_k'=\si'_{j,k}\si''_{j,k}v_k'v_j'
\pl .\]
Since $\norm{v_j-v_j'}{}=\norm{u_j-u_j'}{}$ for all $j$, we have $\rho(\si,\si') \le \rho(\si\si'',\si'\si'')$. Thus the translation invariance follows by symmetry.

With the translation-invariance, it is sufficient to prove the triangle inequality
\begin{align}\rho(\si', \si'')\le \rho(\si', \si)+\rho(\si, \si'') \label{tri} \end{align}
for all triple $(\si,\si',\si'')$ with a fixed $\si$. Indeed for any $\eta, \si,\si',\si''\in \T(d)$, the triangle inequalities \eqref{tri} for $(\si',\si,\si'')$ and $(\si'\eta,\si\eta,\si''\eta)$ are equivalent. Choosing $\theta \in \A(d)$ such that $\theta\Z^d\cup \Z^d=\{{\bf 0}\}$, then $A_\theta$ is simple. We claim that any two $d$-tuples of unitaries
$(u_1,u_2,\cdots, u_d)$ and $(v_1,v_2,\cdots, v_d)$ on $H$ satisfying the commutation relations of $A_\theta$ are approximately unitarily equivalent, i.e. there exists a sequence $\{w_n\}$ of unitaries on $H$ such that for all $j$
\[\norm{w_nu_jw_n^*-v_j}{}\rightarrow 0 \pl.\]
This can be shown, as in Proposition 4.2 of \cite{HR}, by Voiculescu's noncommutative Weyl-von Neumann Theorem \cite{Voiculescu}. Consider the two canonical $*$-homomrophisms $\pi_u,\pi_v: A_\theta \to B(H)$,
\[ \pi_u(\tilde{u}_j)=u_j \pl , \pl \pi_v(\tilde{u}_j)=v_j\pl , \pl j=1,\cdots,d,\]
where $\tilde{u}_j$'s represent the generators of $A_\theta$. Denote by $\K$ the ideal of compact operators on $H$. We need to verify that $\pi_u^{-1} (\K)\subset \ker \pi_u$ and $\pi_u^{-1} (\K)\subset \ker \pi_v$. $\pi_u^{-1} (\K)$ and $\pi_u^{-1} (\K)$ are proper ideals in $A_\theta$, and hence both are trivial because $A_\theta$ is simple.

Now choose $\si$ such that $e^{2\pi i \theta_{jk}}=\si_{jk}$. For any $\si'$ and $\si''$, find $d$-tuples $(u_1,\cdots, u_d)$ and $(u_1',\cdots, u_d')$ of unitaries on $H$ such that
\[u_ju_k=\si_{jk}u_ku_j \pl , \pl u'_ju'_k=\si'_{jk}u'_ku'_j, \pl \text{and}\pl\pl \pl \max_j \norm{u_j-u_j'}{}\le \rho(\si,\si')+\frac{\epsilon}{2}\pl , \]
 and also $(v_1,\cdots, v_d)$ and $(v_1'',\cdots, v_d'')$ such that
 \[v_jv_k=\si_{jk}v_kv_j \pl , \pl v''_jv''_k=\si''_{jk}v''_kv''_j, \pl \text{and} \pl\pl\pl \max_j \norm{v_j-v_j''}{}\le \rho(\si,\si'')+\frac{\epsilon}{2}\pl .\]
Since $(u_1,\cdots, u_d)$ and $(v_1,\cdots, v_d)$ are approximately unitarily equivalent, there exists a unitary $w$ on $H$ such that
\[\max_j \norm{wu_jw^*-v_j}{}\le \epsilon .\]
Then take $\bar{u_j}=wu_j'w^*$, we have
\begin{align*}\rho(\si',\si'') &\le \max_j \norm{\bar{u_j}-v_j''}{}\le \max_j (\norm{wu_j'w^*-wu_jw^*}{} +\norm{wu_jw^*-v_j}{}+ \norm{v_j-v_j''}{})\\& \le \max_j \norm{wu_j'w^*-wu_jw^*}{} +\max_j \norm{wu_jw^*-v_j}{}+ \max_j\norm{v_j-v_j''}{}\\& \le \rho(\si',\si)+\rho(\si,\si'')+2\epsilon \pl.\end{align*}
Therefore we prove the triangle inequality.

Finally, the first inequality of \eqref{final} is a direct consequence of Proposition 4.6 of \cite{HR}. On the other hand, let $\theta,\theta'\in \A(d)$ (we may assume $\theta$ nonsingular by translation invariance) such that $\si_{jk}=e^{2\pi i \theta_{jk}}$ and $\si'_{jk}=e^{2\pi i \theta'_{jk}}$. We have $P_1,P_2,\cdots, P_d$ and $P'_1,P'_2,\cdots, P'_d$ be the self-adjoint operators from Theorem \ref{boundedd}. The second inequality follows from choosing
\[u_j(t)=e^{\sqrt{2\pi}iP_jt}\pl , \pl u'_j(t)=e^{\sqrt{2\pi}iP'_jt}\pl, \pl j= 1, \cdots ,d \pl.\]
\end{proof}

%\subsection{Part to be discussed}
%It is proved that
%\[d(\theta, \theta')=\inf \max\{u(\theta)-u(\theta')}\]
\noindent\emph{Acknowledgements}---I thank my advisor Marius Junge for suggestions and guidance on this research project.
I am grateful to Florin P. Boca for passing me the problem and helpful discussions, and to Stephen J. Longfield for numerous corrections.

%\bibliographystyle{abbrv}
%\bibliography{OA}

\section{Appendix: Projective unitary representations and Twisted group $C^*$-algebras}
In this appendix, we provide an argument for the universality of
the noncommutative Euclidean space $E_\theta$ defined in Section 3. One can identify $E_\theta$ as a \emph{twisted group $C^*$-algebra} and recall its natural connection to projective unitary representation. We refer to the survey \cite{Mackey} for more information about this topic.

Let $G$ be a locally compact Hausdorff group and $e$ be the identity of $G$. A strongly continuous map $u: G \to U(H)$ is a \emph{projective unitary representation} if there exists a (continuous) function $\si: G\times G \to \T$ such that
\[u(g)u(h)=\si(g,h)u(gh)\pl ,\pl g,h\in G \pl.\]
The function $\si$ is called the multiplier associated to $u$ and $u$ is called a $\si$-representation. It follows from the group structure that for all $g,g_1,g_2\in G$
\begin{enumerate}
\item[i)]$\si(g,e)=\si(e,g)=1$;
\item[ii)]$\si(g,g_1)\si(gg_1,g_2)=\si(g,g_1g_2)\si(g_1,g_2)\pl.$
\end{enumerate}
A function $\si: G\times G \to \T$ satisfying i) and ii) is called a $2$-cocycle of $G$ with values in $\T$.

Given a $\T$-valued $2$-cocycle $\si$ of $G$, the Banach $*$-algebra $L_1(G, \si)$ is defined as the set $L_1(G)$ equipped with the $\si$-twisted convolution and involution given by
\[f_1*_\si f_2(g)=\int_G f_1(g_1)f_2(g_1^{-1}g)\si(g_1,g_1^{-1}g)d\mu(g_1) \]
where $d\mu$ is the (left) Haar measure on $G$,
and $f^*(g)=\overline{\si(g,g^{-1})}f(g^{-1})$. $L_1(G, \si)$ can be represented on $L_2(G,\mu)$ as follows,
\[\la_\si(f) (h)= f*_\si h \pl ,  f\in L_1(G), h\in L_2(G) \pl.\]
This is called the left $\si$-regular representation of $G$.
The reduced $\si$-twisted group $C^*$-algebra, denoted by $C^*_r(G,\si)$, is the norm closure of $L_1(G,\si)$ in $B(L_2(G))$. The full $\si$-twisted group $C^*$-algebra $C^*(G,\si)$ is defined as the enveloping $C^*$-algebra of $L_1(G,\si)$. There is an one-to-one correspondence between $\si$-representations of $G$ and representations of $C^*(G,\si)$. If $G$ is amenable, $C^*(G,\si)$ is isomorphic to $C^*_r(G,\si)$ and the left $\si$-regular representation of $C^*(G,\si)$ on $L_2(G,\si)$ is faithful.

Back to the noncommutative Euclidean space $E_\theta$, a symplectic bilinear form $\theta$ introduces a $2$-cocycle of $\R^d$ as follows
\[\si_\theta(\bs,\bt)=\exp(\frac{i}{2}\theta(\bs,\bt)) \pl, \pl \bs,\bt\in \R^d \pl.\]
The Moyal product $\star_\theta$ is the Fourier transform of $\si$-twisted convolution. One identifies $E_\theta=C^*_r(\R^d, \si_\theta)$ and it is further isomorphic to $C^*(\R^d, \si_\theta)$ because $\R^d$ is amenable. Thus there is an one-to-one correspondence between $*$-homomorphism from $E_\theta$ and $\theta$-representation of $\R^d$. One can use an alternative argument by identifying $E_\theta$ with an iterated crossed product $C_0(\R)\rtimes \R\rtimes \cdots \rtimes \R$, which uses the amenablity of $\R$.

\end{document}